\documentclass[12pt]{article}
\usepackage{amsmath,amsfonts,amssymb}
\usepackage{comment}
\usepackage{txfonts}       % selects Times Roman as basic font
\usepackage{helvet}         % selects Helvetica as sans-serif font
\usepackage{courier}        % selects Courier as typewriter font
\usepackage{a4wide}        % activate if the above 3 fonts are
 \usepackage{authblk}
 \usepackage{tikz}
\definecolor{lblue}{rgb}{0.8,0.8,1}
\definecolor{lgray}{rgb}{0.85,0.85,0.85}

\usepackage{color}
\usepackage{array}
\usepackage{todonotes}
\usepackage{makeidx}         % allows index generation
\usepackage{graphics,graphicx}        % standard LaTeX graphics tool
\usepackage{multicol}        % used for the two-column index
\usepackage[bottom]{footmisc}% places footnotes at page bottom

\newcommand{\bfc}{{\boldsymbol c}}
\newcommand{\bfb}{{\boldsymbol b}}
\newcommand{\bfA}{{\boldsymbol A}}
\newcommand{\bfB}{{\boldsymbol B}}
\newcommand{\bfx}{{\boldsymbol x}}
\newcommand{\bfy}{{\boldsymbol y}}
\newcommand{\bfz}{{\boldsymbol z}}
\newcommand{\bfu}{{\boldsymbol u}}
\newcommand{\bfv}{{\boldsymbol v}}

\newcommand{\bftheta}{{\boldsymbol\theta}}
\newcommand{\bftwei}{{\boldsymbol\vartheta}}
\newcommand{\bfI}{{\boldsymbol I}}
\newcommand{\bfP}{{\boldsymbol P}}
\newcommand{\bfp}{{\boldsymbol \lambda}}
\newcommand{\bfq}{{\boldsymbol \mu}}
\newcommand{\bff}{{\boldsymbol f}}

\newcommand{\bfn}{{\boldsymbol n}}
\newcommand{\bflam}{{\boldsymbol \lambda}}
\newcommand{\bfmu}{{\boldsymbol \mu}}
\newcommand{\bfS}{{\boldsymbol S}}
\newcommand{\bfK}{{\boldsymbol K}}
\newcommand{\bfsig}{{\boldsymbol\sigma}}
\newcommand{\bfeps}{{\boldsymbol\varepsilon}}
\newcommand{\IR}{\mathbb{R}}

\newcommand{\lom}[2]{(#1,#2)_\Omega}
\newcommand{\lga}[2]{\left<#1,#2\right>_{H^{1/2}(\Gamma),H^{-1/2}(\Gamma)}}

\newcommand{\sign}{-}
\newcommand{\mcCP}{\text{D}}

\newtheorem{proposition}{Proposition}
\newtheorem{remark}{Remark}
\newenvironment{proof}{\noindent \newline {\bf Proof.}}
{\hfill \mbox{\fbox{} } \newline}
\numberwithin{equation}{section}
\usepackage{todonotes}
\usepackage{hyperref,algorithm,subfigure}

\date{}             % used for the subject indexpeter.hansbo@ju.se
\begin{document}

\title{The augmented Lagrangian method as a framework for stabilised methods in computational mechanics}
\author[$\dagger$]{Erik Burman}
\author[$\ddagger$]{Peter Hansbo}
\author[$\star$]{Mats G. Larson}
\affil[$\dagger$]{\small Department of Mathematics, University College London, London, UK--WC1E  6BT, United Kingdom}
\affil[$\ddagger$]{\small Department of Mechanical Engineering, J\"onk\"oping University, SE-55111 J\"onk\"oping, Sweden}
\affil[$\star$]{\small Department of Mathematics and Mathematical Statistics, Ume{\aa} University, SE-90187 Ume{\aa}, Sweden}
\maketitle

\abstract{In this paper we will review recent advances in the application of the  
augmented Lagrange multiplier method as a general approach for generating 
multiplier--free stabilised methods. We first show how the method generates 
Galerkin/Least Squares type schemes for equality constraints and then how it 
can be extended to develop new stabilised methods for inequality constraints. 
Application to several different problems in computational mechanics is given.}

\section{Introduction}
\label{sec:1}
The Augmented Lagrangian Method (ALM) has a long history in optimisation. In its standard form it can be seen as
augmenting standard Lagrange multiplier methods with a penalty term, penalising the constraint equations. It was introduced in order to combine the advantages of the penalty method and the multiplier method in the context on constrained optimisation independently by Hestenes and Powell in \cite{Hest69, Pow69}. It was then extended to the case of optimization with inequality constraints by Rockafellar in \cite{Ro73, Ro73b}.
Soon afterwards the potential of ALM for the numerical approximation of partial differential equations (pde) and computational mechanics was explored in Glowinski and Morocco \cite{GM75} and by Fortin in \cite{Fort77}. For overviews of the early results on augmented Lagrangian methods for approximation of pde we refer to the monographs by Glowinski and coworkers \cite{FG83,GlLeTa89}.

In computational mechanics,  Lagrangian methods have the drawback of having to fulfil an {\em inf--sup}\/ condition to ensure
stability of the discrete scheme such that the balance between the discretisation of the primal variable and the multiplier variable must be
chosen carefully. Adding a penalty term does not change this situation, and in computational mechanics ALM has therefore 
been used mostly in an iterative approach (improving the conditioning of the discrete system) \cite{FG83,GlLeTa89,SiLa92,WrZa93,LaOa94,ZaWrSc95,WhWiWo14}, or as a way of strengthening control of the constraints in cases where the 
discretisation is under-constrained. It was also shown to improve convergence in some cases by making the penalty parameter mesh dependent in \cite{BoLo97}. Recently similar ideas have been applied in the context of preconditioning solution methods for discretisations of incompressible flows \cite{FMW19, OZ22}. The ideas of extending the ALM to variational inequalities of \cite{Ro73, Ro73b} were introduced in the context of contact mechanics by Alart and Curnier in \cite{AC91}. 

An early approach to weak boundary conditions for finite element methods was introduced orignally by Nitsche in \cite{Nit70}, using a method that is related to ALM, but without any multiplier. Indeed here the multiplier has been replaced by its physical representation, the normal boundary flux. Only recently 
this possibility of substituting the multiplier by its physical interpretation in the discrete augmented Lagrangian formulation has been explored in its generality. This approach gives rise to
schemes that are formally equivalent to stabilised Lagrange multiplier methods, where the stabilisation is of Galerkin/Least Squares (GLS) type \cite{St95}.

There is, however a crucial difference between the ALM and GLS stabilisation method, and that is the treatment of variational inequalities.
The classical GLS formulation for variational inequalities of Barbosa and Hughes \cite{BH92} is very close to standard multiplier schemes, whereas 
the ALM supplies an alternative way to define the stabilisation mechanism which transforms the 
variational inequalities to nonlinear {\em equalities}\/ to which iterative schemes can be readily applied. 

There is a very large literature on on variational inequalities in pde and we can not survey the whole field herein. Below we will focus on works on finite element method formulation and error analysis. For theoretical background material relevant to the material herein we refer to \cite{DL76, KS80, Eck05} and for a review of computational aspects including design of special finite element spaces, adaptive method and solvers we refer to \cite{Wohl11} and references therein. 

The theoretical foundation for finite element approximation of variational inequalities was laid in the seminal works by Falk \cite{Falk74, Falk75}, by Brezzi et al. \cite{BFH77,BFH78} and Haslinger \cite{Has79}. For early overviews on computational aspects we refer to the monographs by Glowinski and co-workers \cite{GLT81, Glow84} and Kikuchi and Oden \cite{KO88}. More recent studies of the numerical analysis of finite element methods for variational inequalities include \cite{Hild00, BenBel00,Chen01,BB01,HL02,BB03,HR12}.
For further work on mixed finite element methods we refer to \cite{HJH82, Scholz87,CHL02,BR03hyb, SBL04,BRS05,Sch11}.
For stabilised finite element methods in the context of variational inequalities see \cite{BH92, HH06, HR10, HRS16, GSV17a, GSV17b, GSV19}.
More recently discontinuous Galerkin methods and other non-conforming methods allowing for polygonal elements have been developed for different types of contact problems \cite{WW10, WWC11,BS12, ZCW15,ZCW17,FHS18,WW20, CEG20}. Another recent development is the application of isogeometric analysis to contact problems \cite{TWH11, DEH15, HCHCB18, ABF19}.
Some results on fourth order problems have been reported in \cite{Has78, Scholz87, BSZ12a, BSZ12b, GP16, GSV19}.
Some early error analyses for augmented Lagrangian finite element methods applied to variational inequalities have been proposed in \cite{Chen01, KKT03}.

Optimal error estimates for the unilateral contact problem however remained elusive and typically required some additional assumptions on the interface between the zones of contact and no contact.
The %stabilisation afforded by the
Nitsche ALM, where the multiplier is replaced by its physical interpretation, was first introduced and analysed for variational inequalities by Chouly and Hild \cite{ChHi13} in the setting of friction free small deformation elastic contact (without
explicit reference to augmented Lagrangians). In this context they also showed optimal error estimates without additional a priori assumptions on the contact set. A similar result for the Signorini problem using a Lagrange multiplier approach (without ALM) was derived in \cite{DH15}. The idea of using ALM with eliminated multiplier for contact problems was then extended to various other models in \cite{Ch14,ChHiRe15,ChHiRe15a,ChHiRe15b,ChMlRe18,ChHiLlRe19}; for an overview, cf. \cite{CFHMPR17}.
Finite element methods using ALM in the form of a nonlinear equality without eliminating the multiplier was analysed in \cite{BuHaLa19}. In the context of non-conforming approximation the approach has been applied in \cite{CCE20} and using IGA in \cite{Fab18, HCHCB18}. It has been explored for CutFEM applications in \cite{FaPoRe16,BuHa17}, for obstacle problems in \cite{BuHaLaSt17,BuHaLa18b}, and for Signorini boundaries in the plate model in \cite{BuHaLa19b}. Typically the analysis of Nitsche's method requires some additional regularity assumptions in order to make sense of the non-conforming terms and we will consider this case below. An analysis for low regularity solutions for Nitsche type methods applied to contact problems was proposed in \cite{GSV19, GSV20}.  The reformulation of the variational inequality as a nonlinear equality with elimination of the multiplier is also advantageous in multi physics applications as illustrated in \cite{BFF20,BFFG22} and to impose positivity in flow problems \cite{BE17}.

Our main objective in this paper is to introduce the ALM in a model context, starting with the original formulation for optimization under constraints and then presenting the extension to pde approximation in an abstract framework. Particular focus will be given to variational inequalities that are rewritten as nonlinear equalities in the ALM framework. Here we prove existence and best approximation estimates
for the multiplier method under the assumption of sufficient smoothness of the multiplier. We discuss stabilised methods and sketch how these results generalize to the case where the multiplier is eliminated. The versatility of the approach is 
then shown by applying it in some different
settings. 

In section \ref{sec:finite} we start by recalling the
augmented Lagrangian method in the finite
dimensional setting both for equality and inequality constraints and
derive the augmented Lagrangian formulation for inequality constraints
using the equality constraint formulation and slack variables. In
section \ref{sec:AGLS} we then discuss the use of the augmented
Lagrangian in the context of partial differential equations and
present the properties of the formulation in an abstract framework. We
show how the necessary a priori bounds for existence of discrete
solutions are obtained and we derive best
approximation estimates for the augmented Lagrangian finite element
method.  In section \ref{sec:concrete} we proceed and give a number of
different applications drawing from fluid and solid mechanics. The
paper finishes with some numerical experiments in section
\ref{sec:numerics} showing the versatility of the proposed framework.
\section{The finite dimensional setting}\label{sec:finite}

We begin by recalling the ALM for finite dimensional optimisation problems and by giving an informal introduction to some key ideas to be used in the following.
Below we will frequently use the notation $a \lesssim b$ for $a \leq C b$.

\subsection{Optimisation with equality constraints}\label{sec:findim}
We consider the  quadratic optimisation problem:
\begin{align}\label{eq:opt-eq}
\min_{\bfx \in \IR^n} f(\bfx) \quad \text{subject to} \quad  g_i(\bfx) = 0, \quad i=1,\dots, m 
\end{align}
This problem can be solved by the Lagrange multiplier method, seeking stationary points to the function
\begin{equation}
\mathcal{L}(\bfx,\lambda_1,\ldots,\lambda_m) = f(\bfx) +\sum_{i}\lambda_i g_i(\bfx)\label{lagfunc}
\end{equation}
solving the system of equations
\begin{align}
\nabla f \sign \sum_i\lambda_i \nabla g_i = {}& {\bf 0}\label{lag1}\\
g_i={}& 0, \; i=1,\ldots, m\label{lag2}
\end{align}
It can also be solved approximately by the penalty method, seeking the minimum to the function
\begin{equation}\label{eq:fin_penalty}
\mathcal{L}_\gamma(\bfx) = f(\bfx) +\frac{\gamma}{2}\sum_{i}g_i(\bfx)^2
\end{equation}
where $\gamma\in \IR^+$ is a given (large) penalty parameter.
We note that the penalty method has a strong regularising effect on the problem in the sense that if
some of the side conditions are (close to being) linear combinations of each other, this does not matter; indeed even if $g_j(\bfx) = g_1(\bfx)$ for all $j$ we simply solve
\begin{equation}
\mathcal{L}_\gamma(\bfx) = f(\bfx) +m\frac{\gamma}{2} g_1(\bfx)^2
\end{equation}
which is a well posed problem. This is not the case in the multiplier method, where the system (\ref{lag1})--(\ref{lag2}) would then be ill posed.
The key point is that the side conditions do not come into play explicitly in the penalty method.
\textcolor{black}{On the other hand, in general the minimiser of \eqref{eq:fin_penalty} coincides with that of \eqref{eq:opt-eq} only in the limit as $\gamma \rightarrow \infty$.}
The ALM is a combination of the penalty method and the multiplier method: seek the stationary point to
\begin{align}\label{augfunc}
\mathcal{L}_{\gamma}(\bfx,\bflam) &= f(\bfx) \sign \sum_{i}\lambda_i g_i(\bfx) +\frac{\gamma}{2}\sum_{i}g_i(\bfx)^2
\end{align}
This problem has the same stationary point as (\ref{lagfunc}) and the same stability problem in case of linearly independent side conditions.
We note, however, that the multiplier can be eliminated by first solving (\ref{lag1}), which we symbolically denote by
\begin{equation}\label{multint}
\lambda_i = \frac{df}{dg_i}(\bfx)
\end{equation}
(the multipliers can be interpreted as the change in objective with respect to change in the corresponding side condition), and seek the minimum to the reduced Lagrangian
\begin{equation}
\mathcal{L}_{A}(\bfx) = f(\bfx) +\sum_{i}\frac{\gamma}{2} g^2_i(\bfx)-\frac{df}{dg_i}(\bfx) g_i(\bfx)
\end{equation}
Like in the penalty method, the side conditions are then no longer explicit; however, in case of linear dependence we still have an ill posed problem
in solving (\ref{lag1}) and we cannot obtain the representation (\ref{multint}). But say that we had an alternative way of computing the multiplier so that symbolically we had
\begin{equation}\label{eq:newmult}
\lambda_i^*(\bfx)  \approx \frac{df}{dg_i}(\bfx), \quad \lambda_i^*(\bfx)\; \text{computable}
\end{equation}
Then we could consider  the problem of minimising 
\begin{equation}\label{newapp}
\mathcal{L}_{A}^{*}(\bfx) = f(\bfx) +\sum_{i}\left(\frac{\gamma}{2} g_i(\bfx)^2-\lambda_i^*(\bfx) g_i(\bfx)\right) 
\end{equation}
The accuracy of this method would then depend on the accuracy of the approximation (\ref{eq:newmult}) and the stability of the formulation. A typical situation is that there 
is a constant such that 
\begin{equation}
\sum_i |\lambda_i^*(\bfx)|^2 \leq C f(\bfx)
\end{equation}
which gives 
\begin{align}
\mathcal{L}_{A}^{*}(\bfx) &= f(\bfx) +\sum_{i}\left(\frac{\gamma}{2} g_i(\bfx)^2-\lambda_i^*(\bfx) g_i(\bfx)\right)
\\
  &= f(\bfx) +\sum_{i}\left(\frac{\gamma}{2} g_i(\bfx)^2-
  \delta |\lambda_i^*(\bfx)|^2 - \frac{1}{4 \delta} g_i^2(\bfx)  \right)
 \\
  &\geq  f(\bfx) - 
  \delta \sum_{i}  |\lambda_i^*(\bfx)|^2  +\sum_{i} \left( \frac{\gamma}{2} - \frac{1}{ 4 \delta} \right) g_i^2(\bfx)
\\
  &\geq ( 1- \delta C ) f(\bfx) 
  +\sum_{i} \left( \frac{\gamma}{2} - \frac{1}{ 4 \delta} \right) g_i^2(\bfx)  
\\
&\gtrsim  f(\bfx)  + \sum_{i}  g_i^2(\bfx)  
\end{align}
where we obtained the last estimate by taking  $\delta$ sufficiently small and $\gamma$ sufficiently large. We conclude that the minimization problem for $\mathcal{L}_{A}^{*}(\bfx)$ is well posed if $\gamma > \gamma_C$. This is the basic idea 
that underlies the application of the ALM as a stabilisation method, in cases where the multiplier can be eliminated.

\subsection{Optimisation with inequality constraints}

We consider next a quadratic optimisation problems of the type: 
\begin{align}\label{eq:opt-ineq}
\min_{x \in \IR^n} f(\bfx) \quad \text{subject to} \quad  g_i(\bfx) \leq 0, \quad i=1,\dots, m 
\end{align}

The augmented Lagrangian for this problem proposed by Rockafellar \cite[Equation (7)]{Ro73} (here with $\gamma=2 r$, and with the multiplier chosen negative) takes the form for $\gamma \in \mathbb{R}^+$,
\begin{align} 
\mathcal{L}_{A}(\bfx,\bflam) 
%= {}& f(\bfx) \sign \sum_{i}\lambda_i g_i(\bfx) +\frac{1}{2}\sum_{4 r}g_i(\bfx)^2
%\\
= & f(\bfx) +\frac{1}{2\gamma} \sum_{i}  \left( [ \gamma g_i(\bfx) \sign \lambda_i]_+^2 -\lambda_i^2\right) \label{eq:rewrite}
\end{align}
where $[x]_+ = \max(x,0)$.

Observe that another equivalent reformulation is given by
\begin{align} \label{eq:rewrite2}
\mathcal{L}_{A}(\bfx,\bflam)  = & f(\bfx) \sign \sum_{i}\lambda_i g_i(\bfx) +\frac{\gamma}{2}\sum_{i}g_i(\bfx)^2 - \frac{1}{2\gamma} \sum_{i} [ \gamma g_i(\bfx) \sign \lambda_i]_-^2
\end{align}
where $[x]_- = \min(x,0)$.
This is easily seen by using that $x = [x]_+ + [x]_-$ and hence
\[
[ x]_+^2 = ([x]_+ +[x]_-)^2 - [x]_-^2 - \underbrace{2 [x]_+ [x]_-}_{=0} = x^2 - [x]_-^2
\]
Applying this in \eqref{eq:rewrite} with $x = \gamma g_i(\bfx) \sign \lambda_i$ leads to \eqref{eq:rewrite2}.
In \eqref{eq:rewrite2} we recognise the augmented Lagrangian for the equality constraint \eqref{augfunc} in the first three terms and the last term is the non-linear switch that introduces the inequality constraint.

To see that \eqref{eq:rewrite} is indeed the natural formulation we introduce slack variables $z_i \in \IR_+$ and rewrite (\ref{eq:opt-ineq})  
in the form 
\begin{align}
\min_{(\bfx,\bfz) \in \IR^n \times \IR^m_+} f(\bfx) \quad \text{subject to} \quad g_i(\bfx)+z_i =0, \quad i=1,\dots,m
\end{align}
with corresponding augmented Lagrangian  
\begin{align}
\mathcal{L}_{A}(\bfx,\bfz,\bflam) = f(\bfx) \sign \sum_i \left\{ (g_i(\bfx) + z_i) \lambda_i + \frac{\gamma}{2} (g_i(\bfx) + z_i)^2 \right\}
\end{align}
for which we seek stationary points, minimizing in $(\bfx, \bfz)$.
Here we may now perform the optimization over $\bfz \in \IR_+^m$ explicitly by noting that for each $\bfx$ and $\bflam$ we obtain a sum of quadratic polynomials in $z_i$ of 
the form 
\begin{align}\label{eq:ineq-a}
\sign (g_i(\bfx) + z_i) \lambda_i + \frac{\gamma}{2} (g_i(\bfx) + z_i)^2
=
\frac{1}{2\gamma} \Big( (\gamma (g_i(\bfx) + z_i ) \sign \lambda_i )^2 - \lambda_i^2 \Big)
\end{align}
and therefore the minimum is attained at $\gamma z_i = - (\gamma g_i \sign \lambda_i)$ and taking the constraint $z_i \in \IR_+$ into account we find that $\gamma z_i = [-(\gamma g_i(\bfx) \sign \lambda_i)]_+$. Inserting this expression for $\gamma z_i$ into (\ref{eq:ineq-a}) and using the identity $a+[-a]_+ = [a]_+$ we arrive at 
\begin{align}\label{eq:ineq-b}
\mathcal{L}_{A}(\bfx,\bfz,\bflam) = f(\bfx) + \frac{1}{2\gamma}\sum_i ([\gamma g_i(\bfx) \sign \lambda_i]_+^2 - \lambda_i^2)
\end{align}

Alternatively we may seek stationary points to the standard Lagrangian
\begin{equation}\label{lagfunc2}
\mathcal{L}(\bfx,\lambda) = f(\bfx) \sign \sum_{i}\lambda_i g_i(\bfx)
\end{equation}
under the Karush--Kuhn--Tucker (KKT) conditions 
\begin{align}
g_i\leq {}& 0, \; i=1,\ldots, m\label{KT1}\\
\lambda_i\leq {}& 0, \; i=1,\ldots, m\label{KT2}\\
\lambda_ig_i= {}& 0, \; i=1,\ldots, m\label{KT3}
\end{align}
Noting that the  KKT conditions  (\ref{KT1})--(\ref{KT3}) are equivalent to 
the single statement
\begin{equation}\label{eq:lagsubst}
\lambda_i = \sign [\gamma g_i \sign \lambda_i]_+
\end{equation}
where $\gamma\in\IR^+$ is an arbitrary positive number.
We may then rewrite the Lagrangian in the form
\begin{align}
 f(\bfx) \sign \sum_{i}\lambda_i g_i(\bfx) 
 &=  f(\bfx) \sign \sum_{i}\lambda_i \Big( g_i(\bfx) \sign \frac{1}{\gamma}\lambda_i \Big)
 - \frac{1}{\gamma} \lambda_i^2
 \\
  &=  f(\bfx) + \frac{1}{\gamma} \sum_i   \underbrace{[\gamma g_i(\bfx) \sign \lambda_i]_+ ( \gamma g_i(\bfx) + \lambda_i )}_{[\gamma g_i(\bfx) \sign \lambda_i]_+^2}
 - \frac{1}{\gamma} \lambda_i^2
\end{align}
where we used (\ref{eq:lagsubst}) and the fact that $[a]_+a = [a]_+^2$. The substitutions $\lambda_i \mapsto \lambda_i/2$ and $\gamma \mapsto 2 \gamma$ manufactures the Lagrangian (\ref{eq:ineq-b}).

%Rewriting the augmented Lagrangian (\ref{augfunc}) from the equality case by completing the square we obtain
%\begin{align} 
%\mathcal{L}_{A}(\bfx,\bflam) = {}& f(\bfx) -\sum_{i}\lambda_i g_i(\bfx) +\frac{\gamma}{2}\sum_{i}g_i(\bfx)^2
%\\
%= {} & f(\bfx) +\sum_{i} \left(\frac{1}{2\gamma}(\gamma g_i(\bfx)-\lambda_i)^2 -\frac{1}{2\gamma}\lambda_i^2\right) \label{eq:rewrite}
%\end{align}
%we can now transfer it to the inequality case by the simple change
%\begin{equation}\label{augfuncineq}
%\mathcal{L}_{A}(\bfx,\bflam) := f(\bfx) +\sum_{i} \left(\frac{1}{2\gamma}[\gamma g_i(\bfx)-\lambda_i]_+^2 -\frac{1}{2\gamma}\lambda_i^2\right)
%\end{equation}
Writing the optimality system of (\ref{eq:ineq-b}) results in the system of equations
\begin{align}
\nabla f +\sum_i[\gamma g_i-\lambda_i]_+ \nabla g_i = {}& {\bf 0}\label{lagineq1}\\
[\gamma g_i-\lambda_i]_+ ={}& -\lambda_i, \; i=1,\ldots, m\label{lagineq2}
\end{align}
which is a nonlinear equality problem which explicitly includes the KKT conditions. 

Again, if we can use (\ref{eq:newmult}) we may instead seek the minima to 
\begin{equation}\label{newapp2}
\mathcal{L}_{A}^{*}(\bfx) := f(\bfx) +\frac{1}{2\gamma} \sum_{i} [\gamma g_i(\bfx)-\lambda_i^*(\bfx)]_+^2 -(\lambda_i^*(\bfx))^2 
\end{equation}

\section{Iterative solution using the augmented Lagrangian}
The augmented Lagrangian is possibly most well known as the basis for an iterative algorithm for constrained optimization problems. The stationary points of the functional \eqref{augfunc} can be approximated using the following classical algorithm attributed to Usawa, with the application to augmented Lagrangian methods developed in the works of Glowinski and co-workers \cite{AHU58, GLT81, FG83, GlLeTa89}. Following \cite{FG83} we consider the situation where the model problem is to minimize
\[
J(\bfx) := \frac12 \bfx^T \bfA \bfx - \bfb^T \bfx 
\]
over $\bfx \in \IR^n$ under the constraint $\bfB \bfx = \bfc \in \IR^m$. Here $\bfA \in \IR^{n\times n}$ is symmetric positive definite, $\bfb \in \IR^n$ and $\bfB \in \IR^{m \times n}$. The augmented Lagrangian \eqref{augfunc} then takes the form,
\begin{equation}\label{eq:LAforiter}
\mathcal{L}_{A}(\bfx,\bflam) := \frac12 \bfx^T \bfA \bfx - \bfb^T \bfx + \bflam^T(\bfB \bfx -\bfc) + \frac{\gamma}{2} |\bfB \bf x- \bfc|^2
\end{equation}
\begin{algorithm}(Uzawa's algorithm)
\begin{enumerate}
\item Let $\bflam^0 \in \IR^m$
\item Find $\bfx^n \in \IR^n$ such that $\mathcal{L}_{A}(\bfx^n,\bflam^n) \leq \mathcal{L}_{A}(\bfy^n,\bflam^n)$ for all $\bfy^n \in \IR^n$ 
\item Update the multiplier: $\bflam^{n+1} = \bflam^{n} + \rho_n (\bfB \bfx^n- \bfc), \quad \rho_n > 0$.
\end{enumerate}
\end{algorithm}

We note that step 2 is equivalent to solving the linear system, find $\bfx^n \in \IR^n$ such that
$$
(\bfA + \gamma \bfB^T \bfB) \bfx^n = - \bfB^T \bflam^n + \bfb + \gamma \bfB^T \bfc
$$
The iterates $\bfx^n, \bflam^n$ of the iterative method converges to the saddle point of \eqref{eq:LAforiter} provided the steplength  $\rho_n$ satisfies
\[
0 < \alpha_0 \leq \rho_n \leq \alpha_1 < 2 \left(\gamma + \frac{1}{\beta^2} \right)
\]
where $\beta^2$ is the largest eigenvalue of the matrix $\bfA^{-1} \bfB^T \bfB $ defined by
\[
\beta^2 = \max_{\bfv \ne 0} \frac{|\bfB \bfv|^2}{\bfv^T \bfA \bfv}
\]
For a proof of the convergence result we refer to \cite[Chapter 2, Section 4]{GLT81} or \cite[Chapter 1, Section 2]{FG83}.
%Now, in the case of finite element discretisation of the minimisation of energy functionals with side conditions, the
%discretisation must fulfil a discrete {\it inf-sup}\/ condition guaranteeing the well-posedness of the discrete problem, see, e.g., \cite{BrFo91}, and, as a result,
%linear independence of the discrete side conditions. This limits the kinds of discretisations that can be used in practice, which has
%led to the introduction of stabilisation methods to circumvent the condition. Say now that we have a discretisation that violates the {\it inf-sup}\/ condition. If we want to apply the ALM for the solution of our minimisation 
%problem, we are back in the setting of (\ref{augfunc}) after
%discretisation, and the problem of computing the discrete multipliers will be ill posed. However, for a large class of problems we have a clear path to (\ref{eq:newmult}): in engineering mechanics the multiplier often has a clear physical interpretation, and using this interpretation we can compute a different but well posed discrete version of the multiplier, symbolised by (\ref{eq:newmult}), to use as a proxy for the multiplier symbolised by (\ref{multint}), generated by the {\it inf-sup}\/ unstable discretisation. This leads to a novel approach to stabilised methods which is close to GLS for equality constraints and which is readily extended to inequality constraints.
% 

\section{Augmented Lagrangian methods and Galerkin/Least squa\-res}\label{sec:AGLS}

We now turn to the case where the Lagrangian is a functional taking values in some Sobolev space and the numerical method is obtained by finding the stationary points in a finite dimensional approximation space. Typically we are interested in the discretisation of a problem where some energy is minimised under a constraint. To illustrate this we consider the case with equality constraints. Let $V$ and $H$ denote two Hilbert spaces, with dual spaces $V'$ and $H'$, respectively. Let $F:V \rightarrow
\mathbb{R}$ denote a strictly convex $C^2$-functional and $B:V \rightarrow H$ a
linear operator. We are interested in minimising $F$ under a
constraint defined by $B$.  Given the data $f\in V'$ and $g\in H$ We consider the optimization problem
\begin{equation}\label{eq:contopt}
u = \mbox{arginf}_{v \in V} F(v)  - \left<f,v\right>_{V',V} \mbox{ such that } B u = g.
\end{equation}
The Lagrangian takes the form
\begin{equation}\label{eq:Lag_form_cont}
\mathcal{L}(v,\mu) := F(v)  - \left<f,v\right>_{V',V}-
\left<\mu,Bv-g \right>_{H',H}.
\end{equation}
This problem can be shown to have unique solution under suitable hypothesis on the spaces $V$ and $H$ and the operators $F$, $B$, $f$ and $g$ (see for instance \cite[Chapter 1, Section 2.1, Theorems 2.1 and 2.2]{Lions69}). Augmenting the Lagrangian has no effect on the continuous level, but formally an augmented version of \eqref{eq:Lag_form_cont}, in the spirit of \eqref{augfunc} can be written
\begin{equation}\label{eq:ALabst_cont}
\mathcal{L}_\text{A}(v,\mu) := F(v)  - \left<f,v\right>_{V',V}-
\left<\mu,Bv-g \right>_{H',H} + \frac{\gamma}{2} \|Bv-g\|_H^2.
\end{equation}
The discrete version of the ALM based on \eqref{eq:ALabst_cont}, would then be obtained by restricting $\mathcal{L}_\text{A}$ to finite dimensional spaces. As we saw in the previous section the ALM on the discrete level combines the control of the constraint given by the Lagrange multiplier and of the penalty. It also gives us an iterative procedure to find the minimiser. When using the ALM in the context of pde problems the ALM also gives enhanced control of the side condition in the sense of a GaLS method, or a variational multiscale method. To see this we assume that $H=H'=L^2$ and that $H'_h \subset H'$, $V_h \subset V$ are some finite dimensional approximation spaces. Here $h$ denotes the characteristic lengthscale (or mesh parameter) of the discrete space. We let $\pi_H:H \mapsto H'_h$ denote the $L^2$-orthogonal projection onto $H'_h$. Since 
\[
\left<\mu_h,Bv-g \right>_{L^2} = \left<\mu_h,\pi_H (Bv-g) \right>_{L^2}
\]
we see that the Lagrange multiplier only gives control of the projection of $Bv-g$ on the finite dimensional subspace $H_h'$. This may be insufficient for the stability of the method, in particular since $H_h'$ may need to be chosen small compared to $V_h$ for stability reasons, i.e. to satisfy the inf-sup stability condition that we will discuss below. A classical example is the stability of the incompressibility constraint (in which case $B$ is the divergence operator) of the Brinkman problem when the viscosity becomes negligible. Adding the term $\|Bv-g\|_H^2$ enhances the stability,
by adding control of $(I - \pi_H)(Bv-g)$ compared to the pure Lagrange multiplier method. This also shows that a sufficient stabilization can be achieved by augmenting with $\|(I - \pi_H)(Bv-g)\|_H^2$. This we recognise as a stabilization of the orthogonal subscales, which is a member of the family of variational multiscale methods. Of course in the associated Euler-Lagrange equations these terms take the form of GLS stabilizations of some residual quantities. Indeed a number of ideas from the field of stabilized methods can be made to bear to the ALM, but we will not explore this further herein. Instead we will show in the examples below how the design of finite element methods using the ALM allows us to recover some well known GLS methods from computational mechanics.

We can discern two different situations for the continuous problem \eqref{eq:ALabst_cont}:
\begin{enumerate}
\item[A.] The multiplier has enough regularity to define a scalar product with the side condition.
\item[B.] The multiplier has only regularity enough to support a duality pairing with the side condition.
\end{enumerate}
In the first case we can use an analogue to the reformulation  (\ref{eq:rewrite}) which is convenient for the treatment of inequality conditions, and formulate the problem on the continuous level; in the second case this is not formally correct. Indeed if the multiplier does not have sufficient regularity the augmented continuous formulation does not lead to a well-defined problem, unless the augmentation is taken in the continuous $H$-norm, which may be inconvenient from computational standpoint. In this case the reformulation (\ref{eq:rewrite}) is not available. We emphasize that this is  not a problem
in the discrete setting since we can use norm equivalence of discrete spaces to obtain an ALM  that has the right asymptotic scaling. However in order to carry out a rigorous numerical analysis of the resulting finite element method the assumption of additional regularity of the exact solution must be justifiable. This is often, but not always the case. In that sense ALM methods in the situation B can be seen as a non-conforming method.

For the discrete as well as the continuous problem we have two further cases:
\begin{enumerate}
\item[C.] The multiplier has a physical interpretation in terms of the primal variable.
\item[D.] The multiplier cannot be interpreted (or be easily interpreted) in terms of the primal variable.
\end{enumerate}
For the discrete case, we also have the problem of finding suitable approximations to fulfil a discrete {\em inf--sup}\/ condition.
In case C we can use a trick analogous to that of (\ref{eq:newmult}), which gives a class of problems where the
multiplier has been eliminated beforehand; alternatively, the multiplier can be retained and stabilised by the addition of a GLS term, in the spirit of \cite{BaHu91, BH92}. These approaches
give stability without balancing the discretisation of the multiplier space and the space for the primal variable. In case D the multiplier has to be retained, but the inequality case can still be handled in the same way as above and stabilisation is still possible, for instance using interior penalty stabilization where the stabilization acts on the multiplier alone \cite{BuHa10, BuHaLa19}.

\subsection{Abstract framework}

%We first give an abstract framework for the stability analysis of \eqref{eq:ALabst}. 
Since the 
rationale of the method is from numerical approximation we will only
consider formulations that work in the finite dimensional setting,
then A and B above are treated similarly. However it is only in case A
that the discussion holds also for the continuous case.
The resulting numerical methods can be shown to be optimally converging for sufficiently smooth exact solutions, but the problem of convergence is not established for exact solutions that has no additional regularity. 
The question of how to design methods that are valid formulations also for the original pde problem is subtle and requires the design of sophisticated stabilization operators, for an interesting work in this direction we refer to \cite{Bert00}. Below we let $\left<\cdot,\cdot \right>$ denote the $L^2$ scalar product over the the domain of definition of functions in $H$ and we denote the associated norm $\|v\| := \left<v,v\right>^{\frac12}$.

We are interested in minimising $F$ under a
constraint defined by $B$, either as an equality or an inequality
constraint. We will now introduce some sufficient conditions for the abstract analysis below to hold. We will then in the examples show that the assumptions are verified.
\begin{enumerate}
\item We  assume that the operator $B$ is bounded and surjective from 
$V$ to $H$, so that for every $\zeta\in H$ there exists $\xi \in V$
such that $ B \xi = \zeta$ and  $\|\xi\|_V \leq C \|\zeta\|_{H}$. It follows that there exists $\alpha>0$ such that for every
$\mu \in H'$ there holds
\begin{equation}\label{eq:cont_infsup}
\alpha \|\mu\|_{H'} \leq \sup_{v \in V} \frac{\left<B v, \mu \right>_{H,H'}}{\|v\|_{V}}
\end{equation}
\item We also assume that $V_h$ and $H_h'$ are chosen in such a way that this property carries over to the finite dimensional setting, in the sense that a so called Fortin interpolant exists, for all $v \in V$ such that $B v \in H$, there exists $\exists i_F v \in V_h$ such that for all $q_h \in H_h'$,
\begin{equation}\label{eq:Fortin_interp}
\left<B (v - i_F v), q_h\right> = 0, \quad \|i_F v\|_V+ \|B i_F v\|_{H_h} \lesssim \|v\|_V + \|B v\|_{H_h}
\end{equation}
Note that for $v \in V_h$ there holds $i_F v = v$. 
\item  We assume that the surjectivity also holds for the discrete spaces on the following form: for all $\mu_h \in H_h'$ there exists $v_h \in V_h$ such that for all $q_h \in H_h'$,
\begin{equation}\label{eq:disc_infsup}
\left<\mu_h -B v_h, q_h\right>= 0, \quad \|v_h\|_V + \|B v_h\|_{H_h} \lesssim  \|\mu_h\|_{H_h}
\end{equation}
\end{enumerate}
Discrete surjectivity is a consequence of the discrete inf-sup condition which typically is equivalent with the existence of the Fortin interpolant \cite[Lemma 26.9]{EG22b}. We state both \eqref{eq:Fortin_interp} and \eqref{eq:disc_infsup} separately here for future reference and to highlight the difference of the norms required in the right hand side. If we are in a non-conforming situation it is not immediately clear that equivalence holds. Note however that if the spaces are such that $\|B v - \pi_{H_h} B v\|_{H_h} \leq \|v\|_V$ then \eqref{eq:disc_infsup} implies \eqref{eq:Fortin_interp}. 
%On the other hand if the $L^2$-projection $\pi_H: H \mapsto H'_h$ is surjective, then %\eqref{eq:Fortin_interp} implies \eqref{eq:disc_infsup}. But the stability is not OK?

The form of the stabilities in \eqref{eq:Fortin_interp} and \eqref{eq:disc_infsup} appear a bit ad hoc here, but as we shall see below this is the natural stability to require for the analysis. Here the
% discrete form $\left<\cdot, \cdot\right>_{H_h,H'_h}$ and 
norm  $\|\cdot\|_{H_h}$ is an $h$-weighted $L^2$-norm and will be discussed below. 
%  are $L^2$-based and associated norms $\|\cdot\|_{H'}, \, \|\cdot\|_{H'_h}$ will be defined
% by introducing mesh-dependence below. 
%For the finite element analysis it
%is natural to introduce the triple norm on $(v,\mu) \in V \times H$
%\[
%|||(v,\mu)|||_h := \|v\|_V + \|\mu\|_{H_h'}.
%\]
%\textbf{Assume existence of Fortin interpolant and deduce \eqref{eq:infsup} from this. Then use Fortin interpolant in error analysis for r_h u and L2 projection for \pi_h \lambda.}
\subsection{Equality constraints}
We wish to solve the optimization problem \eqref{eq:contopt} and recall the formal 
augmented Lagrangian similar to \eqref{augfunc} given by 
\begin{equation}\label{eq:ALabst}
\mathcal{L}_\text{A}(v,\mu) := F(v)  - \left<f,v\right>_{V',V}-
\left<\mu,Bv-g \right>_{H',H} + \frac{ \gamma}{2} \|Bv-g\|_H^2
\end{equation}
%Adding and subtracting $\frac{1}{2 \gamma} \left<\mu,\mu\right>_H$ and
%we have
%\[
%\mathcal{L}_\gamma(v,\mu) := F(v)  - \left<f,v\right>_{V',V}-
%\left<\mu,Bv-g-\frac{1}{2 \gamma} \mu \right>_H - \frac{1}{2 \gamma} \left<\mu,\mu\right>_H+ \frac{\gamma}{2} \|Bv-g\|_H^2
%\]
For later use with inequality constraints, we would now like to use the analogy to (\ref{eq:rewrite}).
However, this is not possible unless $H' = H := L_2$,  where $L_2$ denotes the space of square integrable functions over the pertinent domain, which is case A above. In this particular case,  
completing the square, $-2 a b + b^2 = (a-b)^2 - a^2$, results in the following equivalent formulation
\begin{equation}\label{eq:absequ-a}
\mathcal{L}_\text{A}(v,\mu) := F(v)  + \left<f,v\right>_{V',V}+
\frac{\gamma}{2} \|Bv-g -\frac{1}{2 \gamma} \mu \|^2 - \frac{1}{2 \gamma} \|\mu\|^2
\end{equation}
analogous to (\ref{eq:rewrite}). We let the semi-linear form $a:V
\times V \rightarrow \mathbb{R}$
be defined by the Gateaux derivative of $F(v)$,
\begin{equation}
a(u;v) := \left<\frac{\partial F}{\partial u}(u), v \right>_{V',V}
\end{equation}
and we assume that the form $a$ satisfies the %growth condition
%\begin{equation}\label{eq:growth}
%\frac{a(v;v)}{\|v\|_V} \rightarrow \infty \mbox{ if } \|v\|_V \rightarrow \infty.
%\end{equation}
positivity, monotonicity and continuity conditions 
\begin{gather}\label{eq:aprop1}
a(v;v) \geq \alpha \|v\|_V^2, \qquad \alpha>0
\\
a(w_1;w_1-w_2) - a(w_2;w_1 - w_2 ) \ge \alpha \|w_1 - w_2\|_V^2 \label{eq:aprop2}\\
|a(w_1;v) - a(w_2;v)|\leq C \|w_1-w_2\|_V \|v\|_V \label{eq:aprop3}
\end{gather}
The optimality system obtained by differentiating (\ref{eq:absequ-a}) then reads: 
find $(u,\lambda) \in V \times H'$such that
\begin{eqnarray}
a(u;v) - \left<\lambda, Bv\right>_{H',H} - \left<\mu, Bu\right>_{H',H} +
  {\gamma} \left<Bu, Bv\right>_H = \left<f,v_h\right>_{V',V} +
  \left<g, \mu + {\gamma} Bv \right>_H
\end{eqnarray}
for all $(v,\mu)\in V\times H'$. Here we simply replace $V$ and $H'$ by $V_h$ and $H_h'$ to obtain the discrete method.

We also want to handle case B. Then typically $Bv\in H := H^r$ where $H^r$ denotes a (potentially fractional) Hilbert space with $r > 0$, 
and consequently $\mu\in H' := H^{-r}$, the dual to $H^r$. Since $H^{-r} \not \subset H^r$ the formulation \eqref{eq:absequ-a} no longer makes sense. Instead in the spirit of discretize first then optimize we move to the discrete counterpart of (\ref{eq:contopt}) and introduce
discrete spaces $V_h\subset V$ and $H'_h \subset H'$. The finite element method then amounts to seek stationary points in $V_h$ and $H'_h$ to the augmented Lagrangian (\ref{eq:ALabst}).
On the finite dimensional finite element spaces we can approximate the continuous norms $\|\cdot\|_H$ and $\|\cdot\|_{H'}$ by discrete counterparts
\begin{equation}\label{eq:discrete1}
\| Bv \|_H^2 \approx  \| Bv \|_{H_h}^2 := \| h^{-r} Bv \|_{L_2}^2
\end{equation}
and
\begin{equation}\label{eq:discrete2}
\| \mu \|_{H'}^2 \approx  \| \mu \|_{H'_h}^2 := \| h^{r} \mu \|_{L_2}^2
\end{equation}
where $h$ is the local meshsize (assumed constant in the following for simplicity) and $r\geq 0$ depends 
on the space $H$; loosely speaking $r$ corresponds to the number of derivatives present in the norm 
$\Vert\cdot\Vert_H$. It is also immediate by the Cauchy-Schwarz inequality that the following discrete duality property holds
\[
\left<v,\mu\right>_{H_h,H_h'}  := \left<v,\mu\right> \leq \| v \|_{H_h} \| \mu \|_{H'_h}.
\]
This is done for two reasons
\begin{enumerate}
\item To obtain a well conditioned method, we wish to have the same condition number emanating from the penalty term as from the form $a(\cdot,\cdot)$.
\item The analysis of the resulting methods requires that the discrete norms can be bounded in terms of the form $a(\cdot,\cdot)$ which is only possible if they scale the same way. 
\end{enumerate} 

Now we can use the arbitrariness of $\gamma$ to set
\begin{equation}\label{eq:gamma_scale}
\gamma = \gamma_0/h^{2r}
\end{equation}
where $\gamma_0$ is a problem-- and discretization--dependent constant. Proceeding as above we find that on discrete spaces
\[
\mathcal{L}^h_\text{A}(v,\mu) := F(v)  + \left<f,v\right>_{V',V}+
\frac{\gamma_0}{2h^{2r}} \|Bv-g -\frac{h^{2r}}{2 \gamma_0} \mu \|^2 - \frac{h^{2r}}{2 \gamma_0} \|\mu\|^2
\]
and the discrete optimality system reads: find $(u_h,\lambda_h) \in V_h \times H'_h$
such that
\begin{equation}\label{eq:equconst_FEM}
a(u_h;v) - \left<\lambda_h, Bv\right>- \left<\mu, Bu_h\right> +
  \frac{\gamma_0}{h^{2r}}  \left<Bu_h, Bv\right> = \left<f,v_h\right>_{V',V} +
  \left<g, \mu + \frac{\gamma_0}{h^{2r}} Bv \right>
\end{equation}
for all $(v,\mu)\in V_h\times H'_h$, where $\langle\cdot,\cdot\rangle$ denotes the standard $L_2$ scalar product.
Introducing the global form 
\[
A[(w,\eta);(v,\mu)] := a(w; v) - \left<\eta, Bv\right>- \left<\mu, Bw\right> +
  \frac{\gamma_0}{h^{2r}}  \left<Bw, Bv\right>,
\]
we can cast the optimality system on the compact form: find $(u_h,\lambda_h) \in V_h \times H'_h$
such that
\begin{equation}\label{eq:EulerLag_comp}
A[(u_h,\lambda_h);(v,\mu)] = \left<f,v\right>_{V',V}+
  \left<g, \mu + \frac{\gamma_0}{h^{2r}} Bv \right>.
\end{equation}
for all $(v,\mu) \in V_h \times H'_h$. 

It follows by inspection that any solution to \eqref{eq:contopt} that is sufficiently smooth, i.e. $(u,\lambda) \in V \times H' \cap L^2$ is a solution to \eqref{eq:equconst_FEM} and hence the formulation is consistent. Indeed the stationary point of \eqref{eq:Lag_form_cont} is given by the solution to
\[
a(u;v) - \left<\lambda, B v\right>_{H',H} = \left<f, v\right>_{V',V},\quad \forall v \in V
\]
and
\[
\left< B u , \mu\right>_{H,H'} = \left< g , \mu\right>_{H,H'}.
\]
If the solution is sufficiently regular these equalities hold with $H$ and $H'$ replaced by the $L^2$ norm and we see that in that case the exact solution satisfies the finite element formulation,
\[
\underbrace{a(u;v) - \left<\lambda, Bv\right>}_{=\left<f, v\right>_{V',V}} - \underbrace{\left<\mu, Bu\right> +
  \frac{\gamma_0}{h^{2r}}  \left<Bu, Bv\right>}_{=\left<g, \mu + \frac{\gamma_0}{h^{2r}} Bv \right>} = \left<f,v_h\right>_{V',V} +
  \left<g, \mu + \frac{\gamma_0}{h^{2r}} Bv \right>.
\]

We do not give a full analysis of the linear problem herein, but focus on the nonlinear case in the next section. The analysis immediately also applies to the linear case.
%The continuous version of this problem can be analysed using the arguments of \cite[Chapter 1, Section 2.1, Theorems 2.1 and 2.2]{Lions69}, under certain assumptions on the operator $B$. We postpone this discussion to the next Section .

\subsection{Inequality constraints}

For the subsequent analysis, we will consider the discrete case and
hence we use the space $V_h$ for the primal variable and $H_h'$ for
the dual variable. For simplicity we do not use the subscript $h$ on
all variables below. We wish to solve the continuous optimization problem
\begin{equation}\label{eq:cont_min_ineq}
u = \mbox{arginf}_{v \in V} F(v)  - \left<f,v\right>_{V',V} \mbox{ such that } B u \leq 0
\end{equation}
Where the inequality constraint must be interpreted in the sense of
distributions on $H$ and we will denote the continuous multiplier
appearing in the constrained optimization $\lambda \in H'$.
The weak formulation characterizing the solution to the continuous problem is as follows. Find  $(u,\lambda) \in V \times K$ (where $K := \{ \mu \in H': \mu \leq 0\}$) such that
\begin{equation}\label{eq:stat_ineq}
a(u;v) - \left<\lambda, B v\right>_{H',H} = \left<f, v\right>_{V',V},\quad \forall v \in V
\end{equation}
\begin{equation}\label{eq:stat_ineq2}
\left< B u,  \lambda - \mu  \right>_{H,H'} \leq 0,\quad \forall \mu \in K
\end{equation}

It follows by choosing $\lambda - \mu>0$ in \eqref{eq:stat_ineq2} that $Bu \leq 0$. By taking $\mu=0$ it follows that $\left<B u , \lambda\right>_{H,H'} \leq 0$ and since both $Bu$ and $\lambda$ are negative it follows that $\lambda B u = 0$.

We have arrived at the following Kuhn--Tucker conditions on the multiplier and side condition:
\begin{equation}\label{eq:Kuhn0}
B u \leq 0,\quad \lambda \leq 0, \quad \lambda B u=0.
\end{equation}
We now use the analogue to (\ref{eq:lagsubst}), to show that
(\ref{eq:Kuhn0}) formally is equivalent to 
\begin{equation}\label{eq:lambda0}
\lambda = -{\gamma}\,[Bu-\gamma^{-1} \, \lambda]_+ 
\end{equation}
To derive the finite element formulation we also proceed
formally following the discussion of section \ref{sec:findim} applied to the problem \eqref{eq:cont_min_ineq} with the min taken over the finite dimensional space $V_h$ and write the
augmented Lagrangian, for $\gamma \in \mathbb{R}^+$, $(v,\mu) \in V_h \times H_h'$,

\begin{equation}\label{eq:compact_aug_final}
\mathcal{L}_\text{A}(v,\mu) := F(v)  - \left<f,v\right>_{V',V}
+ \frac{\gamma}2 \|[Bv \sign \mu/\gamma]_+\|^2 -
\frac{1}{ 2\gamma}\|\mu\|^2
\end{equation}
we note that if $\gamma$ is chosen as in \eqref{eq:gamma_scale} we may use \eqref{eq:discrete1} and \eqref{eq:discrete2} to write
\begin{equation}\label{eq:compact_HHprime}
\mathcal{L}_\text{A}(v,\mu) := F(v)  - \left<f,v\right>_{V',V}
+ \frac{\gamma_0}2 \|[Bv \sign \mu/\gamma]_+\|_{H_h}^2 -
\frac{1}{ 2\gamma_0}\|\mu\|_{H_h'}^2
\end{equation}

The finite element optimality system reads: find $(u_h,\lambda_h) \in V_h \times H_h'$
such that
\begin{equation}\label{eq:EulerLag_ineq}
A[(u_h,\lambda_h);(v,\mu)] = \left<f,v\right>_{V',V}
\end{equation}
for all $(v,\mu) \in V_h \times H_h'$,
where
\begin{equation}
A[(w,\eta);(v,\mu)] := a(w; v) + \left<\gamma [Bw \sign \zeta/\gamma]_+ , Bv \sign \mu/
  \gamma\right>- \left<\gamma^{-1} \zeta, \mu \right>.
\end{equation}
Note that in general ($H \ne H'$) and it is not possible to prove well-posedness of \eqref{eq:EulerLag_ineq} in continuous spaces. Nevertheless also in this case a sufficiently smooth solution of the original continuous problem will also be solution to the 
formulation \eqref{eq:EulerLag_ineq}, showing that the formulation remains
consistent.

First we note that for smooth solutions $\lambda \in K$ and \eqref{eq:stat_ineq2} are equivalent to \eqref{eq:lambda0}. Then evaluating \eqref{eq:EulerLag_ineq} at a sufficiently smooth exact solution $(u,\lambda)$ we see that for all $(v,\mu) \in V_h \times H_h'$
\begin{equation}\label{eq:AFEM_form}
A[(u,\lambda);(v,\mu)] := a(u; v) + \underbrace{\left<\gamma [Bu \sign \lambda/\gamma]_+, Bv \sign \mu/
  \gamma\right>}_{= -\left<\lambda, Bv \sign \mu/
  \gamma\right> \,\mbox{by \eqref{eq:lambda0}}} - \left<\gamma^{-1} \lambda, \mu \right> = a(u; v)- \left<\lambda, B v\right>_{H',H}
\end{equation}
and hence by \eqref{eq:stat_ineq} the formulation \eqref{eq:EulerLag_ineq} is consistent for exact solutions $(u,\lambda) \in V \times H'\cap L^2$.

To see the effect of the nonlinear formulation for active and non-active constraints, first assume $[Bw\sign \zeta/\gamma]_+>0$ in \eqref{eq:EulerLag_ineq}. The constraint is active and we see that the equation becomes
\[
a(w; v) - \left<\mu, Bw \right>-  \left< \eta, Bv \right>+ \left<\gamma Bw, Bv \right> = 0
\]
which we recognise as the augmented Lagrangian form from \eqref{eq:equconst_FEM} imposing the 
equality constraint $Bw = 0$. If on the other hand 
$[Bw\sign \eta/\gamma]_+=0$ then the constraint is not
active and the equation \eqref{eq:EulerLag_ineq} takes the form
\[
a(w; v) - \left<\gamma^{-1} \eta, \mu \right> = 0
\]
and we see that $B w$ is free and $\eta = 0$ is imposed. As expected the formulation expresses the conditions of \eqref{eq:Kuhn0} and acts as a nonlinear switch between imposing either $B u = 0$ and $\lambda = 0$.

 Using the parameter $\gamma$ introduced in
\eqref{eq:gamma_scale} and the h-weighted norms introduced in
\eqref{eq:discrete1} and \eqref{eq:discrete2} together with the
inequality $|[a]_+ - [b]_+| \leq |a - b|$ \cite{ChHi13} we see that
the following continuity holds
\begin{multline}\label{eq:Bcontinuity}
\left<\gamma ([Bw_1\sign\eta_1/\gamma]_+ - [Bw_2\sign\eta_2/\gamma]_+ , Bv + \mu/
  \gamma\right> \\
\lesssim (\|B (w_1 - w_2)\|_{H_h} + \|\eta_1 -
\eta_2\|_{H_h'})(\|Bv \|_{H_h} + \|\mu\|_{H_h'})
\end{multline}
Together with \eqref{eq:aprop3} this shows that the form $A$ is continuous.
If $H\equiv
L^2$, the formulation \eqref{eq:EulerLag_ineq} and \eqref{eq:Bcontinuity} makes sense on the
continuous level. Observe that unless $r=0$ the norms are $h$ dependent and hence the bound degenerates for decreasing $h$.

\subsubsection{Stability, existence and uniqueness of solutions}
We will now show that thanks to the properties \eqref{eq:aprop1} -
\eqref{eq:aprop3} we can derive a priori
bounds on $(w,\eta)$ that allows us to prove existence of a solution in the spaces $V_h \times H_h'$, using fixed point arguments. 
\begin{proposition}
Assume that \eqref{eq:cont_infsup}-\eqref{eq:disc_infsup} and \eqref{eq:aprop1}-\eqref{eq:aprop3} hold. Then for every fixed $h$ the formulation \eqref{eq:EulerLag_ineq} admits a unique solution $(u_h,\lambda_h) \in V_h\times H_h'$. The solution satisfies the a priori bound
\begin{equation}\label{eq:apriori_discrete}
\boxed{\|u_h\|_V +
\gamma_0^{\frac12} \|[Bu_h- \gamma^{-1} \lambda_h
 ]_+ + \gamma^{-1}\lambda_h \|_{H_h} + \gamma_0^{-\frac12} \|\lambda_h\|_{H_h'} \lesssim \|f\|_{V'}}
\end{equation}
\end{proposition}
\begin{proof}
If we can show that the operator $A$ is continuous and satisfies a stability condition then
existence follows using Brouwer's fixed point theorem and the arguments of \cite[Chapter 2, Theorem 4.3]{Lions69} (see also \cite[Proposition 4.3]{BuHaLa19} for a discussion of finite element methods and augmented Lagrangian methods). First note that continuity of $A$ follows by \eqref{eq:Bcontinuity} and \eqref{eq:aprop3}. Since $h$ is fixed there is no need for the constant of the continuity to be independent of $h$. Existence of discrete solutions follow from the stability estimate, for all $w,
\eta \in V_h \times H_h'$,
\begin{equation}\label{eq:apriori0}
\boxed{A[(w,\eta),(w + \alpha_\xi \xi,-\eta)] \ge \frac12 \alpha \|w\|^2_V + \frac12
\gamma_0 \|[Bw- \gamma^{-1} \eta
 ]_+ + \gamma^{-1}\eta \|_{H_h}^2 + \frac12 \gamma_0^{-1} \alpha_\xi \|\eta\|_{H_h'}^2}
\end{equation}
where $\xi \in V_h$ is a function such that 
\begin{equation}\label{eq:multstab_partner}
\left<B \xi(\eta) , q_h\right>= -\left< \eta/\gamma  ,
  q_h\right>, \mbox{ for all } q_h \in H_h' \mbox{ and } \|\xi\|_V \lesssim \gamma_0^{-1} \|\eta\|_{H_h'}
  \end{equation}
   (c.f \eqref{eq:disc_infsup}), $\gamma_0\ge 1$ and $\alpha_\xi = 1/2 \min(C_{4.6}^{-2}, C_{4.12}^{-2} \min(1,\gamma_0 \alpha)$ where $C_{4.6}$ and $C_{4.12}$ are the constants in the bounds \eqref{eq:disc_infsup} and \eqref{eq:aprop3} respectively. The bound \eqref{eq:apriori_discrete} follows from \eqref{eq:apriori0} since for a solution $(u_h,\lambda_h)$ there holds
\[
A[(u_h,\lambda_h),(u_h + \alpha_\xi \xi(\lambda_h),-\lambda_h)] = \left<f, u_h + \alpha_\xi\xi(\lambda_h)\right>_{V,V'}
\]
Using the duality pairing we see that
\[
\left<f, u_h + \alpha_\xi\xi(\lambda_h)\right>_{V,V'} \leq \|f\|_{V'}(\|u_h\|_V + \alpha_\xi\|\xi(\lambda_h)\|_V) \lesssim \|f\|_{V'}(\|u_h\|_V + \alpha_\xi\gamma_0^{-1} \|\lambda_h\|_{H_h'})
\]
and the claim follows.

To show \eqref{eq:apriori0} observe that by 
testing with $(v,\mu) = (w,-\eta)$ we have
\begin{equation}
A[(w,\eta),(w,-\eta)] = a(w; w) + \gamma_0^{-1}\|\eta\|_{H_h'}^2
+ \left<\gamma [Bw\sign\eta/\gamma]_+ , Bw + \eta/
  \gamma\right>
\end{equation}
By completing the square we see that
\begin{equation}
\gamma^{-1} \|\eta\|_{L^2}^2 +  \left<\gamma [Bw \sign \eta/\gamma]_+ , Bw + \eta/
  \gamma\right> =  \gamma_0 \|[Bw\sign\eta/
  \gamma]_+ + \gamma^{-1}\eta\|_{H_h}^2
\end{equation}
We conclude that $A$ satisfies the following positivity property, for all
$(w,\eta) \in V_h \times H_h'$,
\begin{equation}\label{eq:positive_A}
A[(w,\eta),(w,-\eta)] = a(w; w) + \gamma_0 \|[Bw\sign\eta/(2
  \gamma)]_+ + \gamma^{-1}\eta \|_{H_h}^2
\end{equation}
Then, since $\eta \in H_h'$ we can use \eqref{eq:disc_infsup} to choose $\xi(\eta) \in V_h$ satisfying \eqref{eq:multstab_partner},
and test with $v =\xi$ and $\mu=0$ to
obtain 
\begin{align}
A[(w,\eta),(\xi,0)] &= a(w;\xi)+ \gamma \left<[Bw\sign\eta/\gamma]_+  
%\gamma^{-1} \eta
,  B \xi(\eta)\right> %+ \gamma_0^{-1} \|\eta\|_{H_h'}^2 
%\\
%& \ge 
%- \gamma^{-1} C^2 /2 \|w\|_V^2 -  \frac12 \gamma_0 \| [Bw+\eta/\gamma]_+
%- \gamma^{-1} \eta\|_{H_h}^2 + \frac12 \gamma_0^{-1} \|\eta\|_{H_h'}^2
\end{align}
Now observe that
\begin{multline}
\gamma \left<[Bw\sign\eta/\gamma]_+,  B \xi(\eta)\right>  =\gamma \left<[Bw\sign\eta/\gamma]_+ +
\gamma^{-1} \eta,  B \xi(\eta)\right> - \underbrace{\left< \eta,  B \xi(\eta)\right>}_{ = -\gamma_0^{-1} \|\eta\|_{H_h'}^2} \\
\ge \gamma_0^{-1} \|\eta\|_{H_h'}^2 - \frac12 C_{4.6}^2 \gamma_0 \| [Bw+\eta/\gamma]_+
+ \gamma^{-1} \eta\|_{H_h}^2 - \frac12 \gamma_0 C_{4.6}^{-2} \| B \xi(\eta)\|_{H_h}^2
\end{multline}
and since $\gamma_0^{\frac12} \|B \xi(\eta)\|_{H_h} \leq C_{4.6} \gamma_0^{\frac12} \|\eta/\gamma\|_{H_h} = C_{4.6} \gamma_0^{-\frac12} \|\eta\|_{H_h'}$
we see that
\[
\gamma \left<[Bw\sign\eta/\gamma]_+,  B \xi(\eta)\right> \ge \frac12 \gamma_0^{-1} \|\eta\|_{H_h'}^2  - \frac12 C_{4.6}^2 \gamma_0 \| [Bw+\eta/\gamma]_+
- \gamma^{-1} \eta\|_{H_h}^2
\]
Combining \eqref{eq:disc_infsup} with \eqref{eq:aprop3} we see that using the boundedness $a(w;\xi) \leq C_{4.12} \|w\|_V \|\eta/\gamma\|_{H_h} \leq C_{4.12} \gamma_0^{-1} \|w\|_V \|\eta\|_{H_h'}$
\begin{multline}\label{eq:positive_eta}
a(w;\alpha_\xi \xi)+ \gamma \left<[Bw\sign\eta/\gamma]_+ +
\gamma^{-1} \eta, \alpha_\xi  B \xi(\eta)\right> \ge - \gamma_0^{-1} \alpha_\xi C_{4.12}^2 \|w\|_V^2 \\-  \alpha_\xi C_{4.6}^2 \gamma_0 \| [Bw+\eta/\gamma]_+
+ \gamma^{-1} \eta\|_{H_h}^2 + \frac12 \gamma_0^{-1} \alpha_\xi \|\eta\|_{H_h'}^2 
\end{multline}
The desired inequality then follow by adding \eqref{eq:positive_A} and \eqref{eq:positive_eta}  for $\gamma_0 \ge 1$ and $$\alpha_\xi = 1/2 \min(C_{4.6}^{-2}, C_{4.12}^{-2}) \min(1,\gamma_0 \alpha).$$

%, satisfying the a priori bound
%\begin{equation}
%\|\lambda_h\|_{H_h'} + \|u_h\|_V \leq \|f\|_{V'}.
%\end{equation}
If $H\equiv L^2$ the analysis can be extended to the continuous case, for details see \cite[Chapter 1, Lemma 4.3]{Lions69}.
%By the convexity and regularity of $F$ and well known properties of
%the operator
%$[\cdot]_+$, \cite[Theorem 3.3]{ChHi13} the operator $A$ is monotone
%and semi-continuous. In the special case that $H = H' = L^2$ it then follows using the arguments %of \cite[Chapter 2, Theorem 2.1]{Lions69} 
%that the the limit of the sequence of $\{u_h,\lambda_h\}$ is a solution to \eqref{eq:EulerLag} %even for inifinite dimensional $V$.

Uniqueness follows in principle from \cite[Chapter 2, Theorem
2.2]{Lions69}, but for completeness we give a simple proof below.
Considering the nonlinearity expressing the constraint we have using
the monotonicity $([a]_+-[b]_+)(a - b) \ge ([a]_+-[b]_+)^2$, and
setting, $e=w_1-w_2$ and $\zeta = \eta_1-\eta_2$,
\begin{align}
&\left<\gamma ([Bw_1\sign \eta_1/\gamma]_+ - [Bw_2\sign \eta_2/\gamma]_+,  Be + \zeta 
  \gamma \right> + \left<\gamma^{-1} (\eta_1 - \eta_2), \zeta  \right> \nonumber
 \\
&\qquad =  \left<\gamma ([Bw_1\sign\eta_1/\gamma]_+ - [Bw_2\sign\eta_2/\gamma]_+, Be \sign \zeta /
 \gamma\right> \nonumber
 \\
&\qquad \qquad  +2 \left<\gamma ([Bw_1\sign\eta_1/\gamma]_+ - [Bw_2\sign\eta_2/\gamma]_+,  \zeta /
  \gamma\right> 
   + \gamma_0^{-1} \| \zeta  \|^2_{H_h'} \nonumber
   \\
   &\qquad
 \ge \gamma_0 \|[Bw_1\sign\eta_1/\gamma]_+ - [Bw_2\sign\eta_2/\gamma]_+ + \gamma^{-1}\zeta  \|_{H_h}^2. \label{eq:Bmono}
\end{align}

It follows from \eqref{eq:aprop2} and \eqref{eq:Bmono} that 
\begin{equation}
\boxed{
E_C[(w_1,\eta_1),(w_2,\eta_2)]^2 +  \alpha \|e\|_V^2 
\leq A[(w_1,\eta_1),(e,-\zeta )] - A[(w_2,\eta_2),(e,-\zeta
  )] 
  } 
  \label{eq:pert_1}
\end{equation}
where $E_C$ is the error in the approximation of the contact zone defined by
\[
E_C[(w_1,\eta_1),(w_2,\eta_2)] := \gamma_0^{\frac12} \|[B w_1+\eta_1/\gamma]_+ -
  [B w_2+\eta_2/\gamma]_+ + \gamma^{-1}\zeta  \|_{H_h}
\]
If we assume that both $\{w_1,\eta_1\}$ and $\{w_2,\eta_2\}$ are
solutions to \eqref{eq:EulerLag_ineq} it follows that the right hand side of \eqref{eq:pert_1}
is zero and 
\[
E_C[(w_1,\eta_1),(w_2,\eta_2)]^2 +  \alpha \|e\|_V^2 = 0.
\]
It follows that $e=0$ and the primal solution is unique. To see that also
the multiplier is unique once again choose $\xi(\eta)$ such that $B
\xi(\eta) = -\eta/\gamma$, in the sense that
$\left<B \xi(\zeta) , q_h\right>= -\left< \zeta /\gamma ,
  q_h\right>$, for all $q_h \in H_h'$,
and test with $v = \xi$ and $\mu=0$, and use
arguments similar as those leading to \eqref{eq:apriori0} to see that
\begin{align}
0 = &A[(w_1,\eta_1),(\xi(\zeta) ,0)] - A[(w_2,\eta_1),(\xi(\zeta) ,0)] \nonumber
\\
&\qquad 
\ge 
-C^2\gamma_0^{-1} \underbrace{\|e\|_V ^2}_{I_1} + \frac12 \gamma_0^{-1} \|\zeta\|^2_{H_h'}  -  \gamma_0 C^2 \underbrace{E_C[(w_1,\eta_1),(w_2,\eta_2)]^2}_{I_2}\label{eq:pert_2}
\end{align}
We have already shown in \eqref{eq:pert_1} that $I_1 = I_2 = 0$ if both $\{w_1,\eta_1\}$
and $\{w_2, \eta_2\}$ are solutions, hence we conclude that $\|\zeta\|_{H_h'} =
0$ which finishes the discussion of (discrete) well-posedness. 
\end{proof}
\subsubsection{Best approximation results}
In this section we will derive a best approximation result for the solution of \eqref{eq:EulerLag_ineq}.
Due to the nonconforming character of the ALM we need to assume that the multiplier is in $H' \cap L^2$. By specifying the approximation properties of our finite element spaces optimal a priori error estimates can be deduced.
\begin{proposition}
Assume that \eqref{eq:cont_infsup}-\eqref{eq:disc_infsup} and \eqref{eq:aprop1}-\eqref{eq:aprop3} hold. Let $(u,\lambda) \in V \times (H' \cap L^2)$ be the solution to \eqref{eq:cont_min_ineq}-\eqref{eq:Kuhn0} and $(u_h,\lambda_h) \in V_h \times H_h'$ be the solution of \eqref{eq:EulerLag_ineq}. Then if $\Phi[(u,\lambda),(u_h,\lambda_h)] := E_C[(u,\lambda),(u_h,\lambda_h)] + \|u - u_h\|_V + \gamma_0^{-\frac12} \|\lambda - \lambda_h\|_{H_h'}$ then there holds

\begin{equation}\label{eq:best_approx1}
\boxed{\Phi[(u,\lambda),(u_h,\lambda_h)] \lesssim \inf_{(v_h,\mu_h) \in V_h \times H_h'} (\|u - v_h\|_V + \gamma_0^{\frac12} \|B ( u - v_h)\|_{H_h} + \gamma_0^{-\frac12} \|\lambda - \mu_h\|_{H_h'})}
\end{equation}
\end{proposition}
\begin{proof}
Since \eqref{eq:pert_1} holds for all $w_1,w_2 \in V$ and $\zeta \in H'\cap L^2$, if the exact solution $u,\lambda \in V \times H'\cap L^2$ we may apply it with $w_1= u$, $\eta_1 = \lambda$ and  $w_2=u_h$, $\eta_2 = \lambda_h$ to obtain, with $e = u - u_h$ and $\zeta = \lambda - \lambda_h$, 
\begin{equation}
E_C[(u,\lambda),(u_h,\lambda_h)]^2 + \alpha \|e\|_V^2
\leq A[(u,\lambda),(e,-\zeta )] - A[(u_h,\lambda_h),(e,-\zeta
  )] \label{eq:error_1}
\end{equation}
Using the consistency of the method we have
\[
A[(u,\lambda),(e,-\zeta )] - A[(u_h,\lambda_h),(e,-\zeta
  )] = A[(u,\lambda),(u - i_F u,\pi_H \lambda - \lambda )] - A[(u_h,\lambda_h),(u - i_F u,\pi_H \lambda - \lambda)]
\]
By the continuity of $a$ we have
\[
a(u;u - i_F u) - a(u_h,u - i_F u) \leq C \|e\|_V \|u - i_F u\|_V  
\]
For the nonlinearity imposing the constraint we notice that by the $L^2$-orthogonality of $\pi_H$,
\[
\gamma^{-1} \left<\zeta,  \lambda - \pi_H \lambda \right> = \gamma_0^{-1} \|\lambda - \pi_H \lambda\|_{H_h'}^2
\]
and using in addition the properties of $i_F u$ we have $\left<\pi_H \zeta,  B ( u - i_F u) + (\lambda - \pi_H \lambda)/\gamma \right> = 0$ and hence using that $\pi_H \zeta = \zeta + \pi_H \zeta - \zeta   = \zeta - (\lambda -\pi_H \lambda)$,
\begin{multline}\label{eq:cont_best_app}
\left<\gamma ([Bu +\lambda/\gamma]_+ - [B u_h +\lambda_h/\gamma]_+), B ( u - i_F u) + (\lambda - \pi_H \lambda)/\gamma \right> \\
= \left<\gamma ([B u+\lambda/\gamma]_+ - [B u_h+\lambda_h/\gamma]_+) + \zeta, B ( u - i_F u) + (\lambda - \pi_H \lambda)/\gamma \right> \\
- \left<\lambda - \pi_H \lambda, B ( u - i_F u) + (\lambda - \pi_H \lambda)/\gamma \right>
\end{multline}
Collecting the above inequalities we obtain using the Cauchy-Schwarz inequality and the arithmetic-geometric inequality in each right hand side,
\begin{multline*}
A[(u,\lambda),(e,-\zeta )] - A[(u_h,\lambda_h),(e,-\zeta
  )] \leq \frac12  (E_C[(u,\lambda),(u_h,\lambda_h)]^2 + \alpha \|e\|_V^2) \\
  + \frac{C}{\alpha} (\|u - i_F u\|_V^2 + \gamma_0 \|B ( u - i_F u)\|_{H_h}^2+ \gamma_0^{-1} \|\lambda - \pi_H \lambda\|_{H_h'}^2)
\end{multline*}
It follows that the following error bound holds,
\[
 E_C[(u,\lambda),(u_h,\lambda_h)]^2 + \alpha \|e\|_V^2  \lesssim \alpha \|u - i_F u\|_V^2 + \gamma_0 \|B ( u - i_F u)\|^2_{H_h}+ \gamma_0^{-1} \|\lambda - \pi_H \lambda\|_{H_h'}^2
\]
By adding and subtracting $v_h$, applying the triangle inequality followed by the stability of the Fortin operator (right inequality of \eqref{eq:Fortin_interp}) there holds
\[
\|u - i_F u\|_V + \gamma_0^{\frac12} \|B ( u - i_F u)\|_{H_h} \lesssim \|u - v_h\|_V + \gamma_0^{\frac12} \|B ( u - v_h)\|_{H_h}
\]
and we conclude using also the definition of the $L^2$-projection $\pi_H$, that
\begin{equation}\label{eq:best_approx2}
\boxed{ E_C[(u,\lambda),(u_h,\lambda_h)] + \|e\|_V \lesssim \inf_{(v_h,\mu_h) \in V_h \times H_h'} (\|u - v_h\|_V + \gamma_0^{\frac12} \|B ( u - v_h)\|_{H_h} + \gamma_0^{-1} \|\lambda - \mu_h\|_{H_h'})}
\end{equation}
Turning to the error in the multiplier we have using \eqref{eq:disc_infsup}
\[
\gamma_0^{-1} \|\pi_H \zeta\|_{H_h'}^2 = \left<\pi_H \zeta, B \xi(\gamma^{-1} \pi_H \zeta) \right> 
\]
where $\xi(\gamma^{-1} \pi_H \zeta)$ is defined by \eqref{eq:disc_infsup} with $z_h = \gamma^{-1} \pi_H \zeta$
using the equation we see that
\begin{multline*}
\gamma_0^{-1} \|\pi_H \zeta\|_{H_h'}^2 =  \left<\pi_H \lambda - \lambda, B \xi(\pi_H \zeta) \right> + \left<\gamma ([Bu +\lambda/\gamma]_+ - [B u_h +\lambda`-h/\gamma]_+) + \zeta, B \xi(\pi_H \zeta)\right> \\
+ a(u; \xi(\pi_H \zeta)) - a(u_h; \xi(\pi_H \zeta))
\end{multline*}
Applying the bound \eqref{eq:aprop3} to the last two terms of the right hand side and the Cauchy-Schwarz inequality to the others and applying the stability of \eqref{eq:disc_infsup} we see that
\[
a(u; \xi(\pi_H \eta)) - a(u_h; \xi(\pi_H \eta)) \leq C_{4.12} \|e\|_V \|\xi(\pi_H \eta)\|_V \leq C_{4.12} \|e\|_V \gamma_0^{-\frac12} \|\pi_H \zeta\|_{H_h'},
\]
\[
\left<\pi_H \lambda - \lambda, B \xi(\pi_H \zeta) \right> \leq \gamma_0^{-\frac12} \|\pi_H \lambda - \lambda\|_{H_h'} \gamma_0^{\frac12} \|B \xi(\pi_H \zeta)\|_{H_h} \leq \gamma_0^{-\frac12} \|\pi_H \lambda - \lambda\|_{H_h'} \gamma_0^{-\frac12} \|\pi_H \zeta\|_{H_h'}
\]
and
\[
\left<\gamma ([Bu +\lambda/\gamma]_+ - [B u_h +\lambda_h/\gamma]_+) + \eta, B \xi(\pi_H \eta)\right>  \leq 
E_C[(u,\lambda),(u_h,\lambda_h)] \gamma_0^{-\frac12} \|\pi_H \zeta\|_{H_h'}.
\]
Collecting terms and dividing through by $\gamma_0^{-\frac12} \|\pi_H \zeta\|_{H_h'}$ we have
\[
\gamma_0^{-\frac12}  \|\pi_H \zeta\|_{H_h'} \lesssim E_C[(u,\lambda),(u_h,\lambda_h)] + \|e\|_V +  \gamma_0^{-\frac12}\|\pi_H \lambda - \lambda\|_{H_h'}
\]
We conclude by applying \eqref{eq:best_approx1} to the right hand side and the triangle inequality $\|\zeta\| \leq \|\lambda - \pi_H \lambda\|+ \|\pi_H \zeta\|$ to obtain,
\begin{equation}\label{eq:best_approx3}
\boxed{\gamma_0^{-\frac12} \|\zeta\|_{H_h'} \lesssim \inf_{(v_h,\mu_h) \in V_h \times H_h'} (\|u - v_h\|_V + \gamma_0^{\frac12} \|B ( u - v_h)\|_{H_h}) + \gamma_0^{-\frac12} \|\lambda - \mu_h\|_{H_h'})}
\end{equation}
The claim now follows by combining \eqref{eq:best_approx1} and \eqref{eq:best_approx2}.
\end{proof}
We observe that the natural norm for $\lambda$ here would be $H'$, but that we here consider the corresponding weighted $L^2$-norm $H_h'$ instead. Since this is an $h$-weighted norm, the resulting $L^2$ error estimate is subotimal compared to approximation. Recovering control of the error in the $H'$ norm would require an additional duality argument that is beyond the scope of this work.
%Finally we can get an estimate in the natural norm $H'$ by using duality,
%\[
%\|\eta\|_{H'} = \sup_{w \in H: \|w\|_H =1} \left<w, \eta \right>_{H,H'} 
%\]
%By adding and subtracting $\pi_H \eta$,
%\[
%\left<B v, \eta - \pi_H \eta \right>_{H,H'} + \left<B i_F v, \pi_H \eta\right>
%\]
\subsubsection{Remark on stabilized methods}
If the discrete spaces $V_h$, $H_h'$ do not satisfy the infsup condition \eqref{eq:Fortin_interp}, one can introduce a stabilization operator $s(\cdot,\cdot)$ which is designed to control the unstable modes.
If a stable pair $V_h$, $\tilde H_h'$, where $\tilde H_h'$ has the same approximation properties as $H_h'$ up to a constant factor, is known, i.e. \eqref{eq:Fortin_interp} and \eqref{eq:disc_infsup} are satisfied for these spaces, then a convenient way of choosing $s$ is by using the following design criteria
\begin{enumerate}
\item Control of unstable modes:
\begin{equation}\label{eq:stab_cont_1}
\gamma_0^{-1/2} \|\mu - \tilde \pi_H \mu\|_{H_h'} \lesssim s(\mu,\mu)^{\frac12}, \quad \forall \mu_h \in H_h' + L^2
\end{equation}
where $\tilde \pi_H$ denotes the $L^2$ projection on $\tilde H_h$.
\item Weak consistency:
\begin{equation}\label{eq:s_weak_cons}
s(\mu - \tilde \pi_H \mu,\mu - \tilde \pi_H \mu ) \sim \gamma_0^{-1} \|\mu - \tilde \pi_H \mu \|_{H_h'}^2, \quad \forall \mu \in L^2
\end{equation}
Here the $\sim$ notation means that the two quantities have the same asymptotics in $h$ for smooth enough $\mu$. 
\end{enumerate}
The simplest choice of $s$ is
\[
s(\eta,\mu) = \gamma^{-1} \left<(\pi_H - \tilde \pi_H) \eta, \mu  \right>
\]
The optimality system of the finite
element formulation then reads: find $(u_h,\lambda_h) \in V_h \times H_h'$
such that
\begin{equation}\label{eq:EulerFEM_stab}
A[(u_h,\lambda_h);(v,\mu)]- s(\lambda_h,\mu) = \left<f,v\right>_{V',V}
\end{equation}
for all $(v,\mu) \in V_h \times H_h'$, with $A$ defined in \eqref{eq:AFEM_form}.

It is then possible to use the monotonicity, the inf-sup stability
\eqref{eq:Fortin_interp} together with \eqref{eq:stab_cont_1} and \eqref{eq:s_weak_cons} to obtain bounds similar to \eqref{eq:best_approx1} for the error of the stabilized Galerkin
  approximation. We only sketch the arguments. The only modification of the stability is that the stabilization operator appears in the left hand side. If $e = u- u_h$ and $\zeta = \lambda-\lambda_h$ then 
 \begin{equation}
E_C[(u,\lambda),(u_h,\lambda_h)]^2 + \alpha \|e\|_V^2+s(\zeta,\zeta)
\leq A[(u,\lambda),(e,-\zeta )] - A[(u_h,\lambda_h),(e,-\zeta
  )] - s(\zeta,-\zeta)
  \label{eq:stab_stab}
\end{equation} 
The key observation to obtain optimal approximation is to use Galerkin orthogonality using $u_h - i_F u$ and $\lambda_h - \tilde \pi_H \lambda$ and then apply a modified continuity estimate. Indeed by the assumptions we have
  $\left<\tilde \pi_H \zeta,  B ( u - i_F u) \right> = 0$ and hence we can modify the continuity \eqref{eq:cont_best_app} the following way,
\begin{multline*}
\left<\gamma ([Bu +\lambda/\gamma]_+ - [B u_h +\lambda_h/\gamma]_+), B ( u - i_F u) + (\lambda - \tilde \pi_H \lambda)/\gamma \right> \\
=\left<\gamma ([Bu +\lambda/\gamma]_+ - [B u_h +\lambda_h/\gamma]_+)+ \tilde \pi_H \zeta , B ( u - i_F u) + (\lambda - \tilde \pi_H \lambda)/\gamma \right> \\
= \left<\gamma ([B u+\lambda/\gamma]_+ - [B u_h+\lambda_h/\gamma]_+) + \zeta, B ( u - i_F u) + (\lambda - \tilde \pi_H \lambda)/\gamma \right> \\
- \left<\zeta -  \tilde \pi_H \zeta  , B( u - i_F u) +(\lambda - \tilde \pi_H \lambda)/\gamma  \right>
\end{multline*} 
where we used that $\left<\tilde \pi_H \zeta  , B( u - i_F u) +(\lambda - \tilde \pi_H \lambda)/\gamma  \right> = 0$.
In this expression all but the last term can be bounded in the same fashion as before. For the last term we apply the Cauchy-Schwarz inequality and then \eqref{eq:stab_cont_1} to see that
\[
\left<\zeta -  \tilde \pi_H \zeta, B( u - i_F u)  + (\lambda - \tilde \pi_H \lambda)/\gamma \right> \leq s(\zeta,\zeta)^{\frac12} (\gamma_0^{\frac12} \|B( u - i_F u)\|_{H_h}+ \gamma_0^{-\frac12} \|\lambda - \tilde \pi_H \lambda\|_{H_h'})
\]
 where now the right hand side is controlled by stability and approximation respectively. This leads to an error estimate for $u - u_h$. The error in the multiplier can also be estimated using that 
$$
\|\zeta\|_{H'_h} \leq \|\tilde \pi_h(\zeta - \zeta_h)\|_{H'_h}+ s(\zeta,\zeta)^{\frac12}
$$
and noting that the first term of the right hand side can be controlled as in the infsup stable case and the second is bounded by \eqref{eq:stab_cont_1}.

\subsection{Eliminating the multiplier}\label{sec:eliminate}
Now we assume that the multiplier can be expressed in the primal
variable through a linear operator $T$ on the continuous level,
i.e. $\lambda = T u$, such that for $v_h \in V_h$, the following inequality that typically is of inverse type, holds
\begin{equation}\label{eq:inv_eq}
 \|T
u_h\|_{H_h'}^2 \leq C_\text{I} \|u_h\|^2_V.
\end{equation}
where  $C_\text{I}$ is a constant that may depend on the mesh geometry, but not on the mesh size. We may then write the Nitsche type form of the equation
\eqref{eq:EulerLag_ineq}:
 find $u_h\in V_h$
such that
\begin{equation}\label{eq:EulerNit}
A[(u_h,T
u_h);(v,T
v)] = \left<f,v\right>_{V',V}
\end{equation}
for all $v \in V_h$,
where $A_h$ was defined in \eqref{eq:AFEM_form}. This formulation,
where the multiplier is eliminated is identified as a nonlinear
GLS method.
For this GLS formulation existence and uniqueness is ensured without any
inf-sup condition \cite[Theorem 3.3]{ChHi13}. Stability is obtained thanks to the continuity of
the $T$ operator, \eqref{eq:inv_eq}.

We now revisit the analysis of the previous section and show that the same results hold for the case when the multiplier has been eliminated.
\subsubsection{Continuity and stability}
We only need to verify \eqref{eq:Bcontinuity} for the method \eqref{eq:EulerNit}. We immediately have for $w_1,w_2,v \in V_h$,
\begin{multline}\label{eq:Nitcontinuity}
\left<\gamma ([Bw_1\sign T w_1/\gamma]_+ - [Bw_2\sign T w_2/\gamma]_+ , Bv + T v/\gamma \right> \\
\lesssim (\|B (w_1 - w_2)\|_{H_h} + \|T (w_1 -
w_2)\|_{H_h'})(\|Bv \|_{H_h} + \|T v\|_{H_h'}) \\
\leq C (\|B (w_1 - w_2)\|_{H_h} + \|w_1 -
w_2\|_{V})(\|Bv \|_{H_h} + \|v\|_{V})
\end{multline}
where we used \eqref{eq:inv_eq} for the second inequality. To prove the a priori estimate that together with the continuity allows for the fixed point analysis we test with $v=u_h$ in \eqref{eq:EulerNit} to obtain using \eqref{eq:aprop1}
\[
\alpha \|u_h\|_V^2 +  \|[B u_h  \sign T u_h/\gamma]_+\|^2_{H_h} - \gamma_0^{-1} \|T u_h\|_{H_h'}^2 \leq A[(u_h,T u_h);(u_h,T u_h)]
\]
Applying \eqref{eq:inv_eq} to the last term of the right hand side we see that
\[
\boxed{(\alpha - C_I/\gamma_0) \|u_h\|_V^2 +  \gamma_0 \|[B u_h  \sign T u_h/\gamma]_+\|^2_{H_h} \leq A[(u_h,T
u_h);(u_h,T
u_h)]}
\]
We conclude that the stability holds for $\gamma_0> C_I/\alpha$. Hence under this condition there exists a discrete solution to \eqref{eq:EulerNit}
\subsubsection{Uniqueness and best approximation estimates}
Uniqueness and best approximation follows using similar arguments, we only detail the best approximation case.
We assume that the exact solution $u$ to \eqref{eq:cont_min_ineq} is sufficiently smooth that
\begin{equation}\label{eq:Nit_consist}
A_h[(u,T
u);(v,T
v)] = \left<f,v\right>_{V',V}, \forall v \in V_h
\end{equation}
Then we may write, $e=u-u_h$ and using the monotonicity of $a$ \eqref{eq:aprop2} and of $[\cdot]_+$ we see that, using the notation $E_C(u,u_h) :=  \gamma_0 \|[Bu \sign T u\gamma]_+ - [B u_h\sign T u_h/\gamma]_+\|_{H_h}^2$,
\begin{equation*}
A_h[(u,T
u);(e,T
e)] - A_h[(u_h,T
u_h);(e,T
e)] \\
\geq \alpha \|e\|_V^2 +  E_C(u,u_h) - \gamma_0^{-1} \|T e\|_{H_h'}^2
\end{equation*}
For the last term of the right hand side observe that
\begin{multline*}
\|T e\|_{H_h'} \leq \|T (u - v_h)\|_{H_h'}+\|T (u_h - v_h)\|_{H_h'} \\
\leq \|T (u - v_h)\|_{H_h'}+C_I^{1/2} \|u_h - v_h\|_{V}\\
\leq \|T (u - v_h)\|_{H_h'}+C_I^{1/2} (\|e\|_{V}+\|u - v_h\|_{V})
\end{multline*}
Hence
\begin{multline}\label{eq:Nit_pert1}
A_h[(u,T
u);(e,T
e)] - A_h[(u_h,T
u_h);(e,T
e)] \\
\geq (\alpha - 3 C_I/\gamma_0) \|e\|_V^2 +  E_C(u,u_h)
- 3 \|T (u - v_h)\|_{H_h'}^2 - 3 C_I \|u - v_h\|_{V}^2
\end{multline}
Fix $\gamma_0 = 6 C_I/\alpha$ so that  $\alpha - 3 C_I/\gamma_0 = \alpha/2$.
Considering the left hand side we have using \eqref{eq:Nit_consist}, for all $v_h \in V_h$
\begin{multline*}
A_h[(u,T u);(e,T e)] - A_h[(u_h,T u_h);(e,T e)]\\ = A_h[(u,T u);(u-v_h,T (u-v_h))] - A_h[(u_h,T u_h);(u-v_h,T (u-v_h))]
\end{multline*}
To conclude we use the continuity \eqref{eq:aprop3} and the arithmetic-geometric inequality,
\[
a(u;u-v_h) - a(u_h;u-v_h) \leq C \|e\|_V \|u-v_h\|_V \leq \frac{\alpha}{4} \|e\|_V^2 + C^2 \|u-v_h\|_V^2
\]
together with the Cauchy-Schwarz inequality and the arithmetic-geometric inequality,
\begin{multline*}
\left<\gamma ([Bu \sign T u\gamma]_+ - [B u_h\sign T u_h/\gamma]_+ , B(u - v_h) + T (u - v_h)/\gamma \right> \\
\leq \frac12 E_C(u,u_h) + \gamma_0 \|B(u - v_h)\|_{H_h}^2 + \gamma_0^{-1} \|T (u - v_h)\|_{H_h'}^2.
\end{multline*}

Applying these inequalities in \eqref{eq:Nit_pert1} we see that for all $v_h \in V_h$
\[
\alpha \|e\|_V^2 +  E_C(u,u_h)\\
\lesssim  \|u-v_h\|_V^2 +  \|B(u - v_h)\|_{H_h}^2  +\|T (u - v_h)\|_{H_h'}^2
\]
Taking square roots of both sides and the infimum over $v_h \in V_h$ in the right hand side we conclude
\begin{equation}\label{eq:best_approx_Nit}
\boxed{E_C(u,u_h) +\|e\|_V  \lesssim \inf_{v_h \in V_h} (\|u-v_h\|_V +  \|B(u - v_h)\|_{H_h}  +\|T (u - v_h)\|_{H_h'})}
\end{equation}
We have sketched a best approximation result for the formulation \eqref{eq:EulerNit}. Observe that no condition needs to be imposed on the finite element space in this case. Instead stability is ensured by the inverse inequality \eqref{eq:inv_eq} that bounds the $H_h'$-norm of the multiplier expressed in the primal variable by the $V$-norm of the primal variable. By equivalence of norms on finite dimensional spaces this bound is always true. The key to optimality of the estimate is the proper $h$-scaling of the discrete norms given in \eqref{eq:discrete1} and \eqref{eq:discrete2}.

We now turn to specific examples.

 \section{Applications}\label{sec:concrete}

 \subsection{The Stokes problem with cavitation}\label{sec:cavitation}

Consider a domain $\Omega$ in ${\mathbb{R}}^n$, $n=2$ or $n=3$ with boundary 
$\partial\Omega$ that is composed of the two subsets $\Gamma_D$ and
$\Gamma_N$ such that $\partial\Omega = \bar \Gamma_D \cup
\bar \Gamma_N$. We consider a lubricant with viscosity $\mu$. The Stokes
equation can then be written 
\begin{equation}\label{eq:Stokes}
-\mu\Delta\boldsymbol u +\nabla p = {\boldsymbol f}\;\text{and}\; \nabla\cdot%
\boldsymbol u = 0\quad\text{in}\;\Omega,
\end{equation}
with $\boldsymbol u=0$ on $\Gamma_D$ and
$(-p\boldsymbol I + \mu\nabla \boldsymbol u)\cdot\boldsymbol n =
\boldsymbol 0$ on $\Gamma_N$. Here, 
$\boldsymbol u$ is the velocity of the lubricant, $p$ is the pressure, and ${%
\boldsymbol f}$ is a force term. The lubricant cannot support subatmospheric
pressure, so an additional condition is $p\geq 0$ in $\Omega$. In order to
incorporate this condition into the model, it can be written as a
variational inequality as follows. Let
\[ 
a(\boldsymbol{u},\boldsymbol{v}) := \int_{\Omega}\mu\nabla\boldsymbol{u}:\nabla\boldsymbol{v}\, {\rm d}\Omega , \quad L(\boldsymbol{v}) := \int_{\Omega} \boldsymbol{f}\cdot\boldsymbol{v}\, {\rm d}\Omega 
\]
and
\begin{equation*}
K=\{p\in L_2(\Omega):\quad p\geq 0\} 
\end{equation*}
Seek $\boldsymbol u\in [H_0^1(\Omega)]^n$ and $p\in K$ such that 
\begin{equation}
a(\boldsymbol{u},\boldsymbol{v})
-\int_{\Omega} p\,\nabla\cdot{\boldsymbol v} \, d\Omega = L(\boldsymbol{v}) , 
\end{equation}
for all ${\boldsymbol v}%
\in [H^1(\Omega)]^n$, and
\begin{equation}
-\int_{\Omega}\nabla\cdot\boldsymbol u\, (q-p) \, d\Omega \leq 0,\quad
\forall q\in K
\end{equation}
To rewrite this problem as a variational equality, we use the Kuhn-Tucker conditions
\begin{equation}\label{kuhntucker2}
p\geq 0, \quad \nabla\cdot\boldsymbol{u} \geq 0, \quad p\,\nabla\cdot\boldsymbol{u} = 0 
\end{equation}
and
again replace conditions (\ref{kuhntucker2}) by the equivalent statement
\begin{equation}\label{lambdadef}
p = \gamma_0[\gamma_0^{-1} p -\nabla\cdot\boldsymbol{u}]_+
\end{equation}
with $\gamma_0$ a positive number. We note here that we can identify the abstract spaces $H$ and $H'$ with $L_2(\Omega)$
and that here the pressure cannot easily be interpreted as coming from a linear operator on the velocity, so we are in cases A and D from Section \ref{sec:AGLS}; the pressure has to be retained but $r=0$ in the discrete norms.

Defining function spaces 
\begin{equation}\label{funcV}
V =\{\boldsymbol{v}\in [H^1(\Omega)]^n: \; \boldsymbol{v} =\boldsymbol{0}\;\text{on $\Gamma^\text{D}$}\}, \quad Q = L_2(\Omega)
\end{equation}
and seeking $(\boldsymbol{u},p)\in V\times Q$ 
%equilibrium requires, with
%\[
%a(\boldsymbol{u},\boldsymbol{v}) := \int_{\Omega}\mu\nabla\boldsymbol{u}:\nabla\boldsymbol{v}\, {\rm d}\Omega , \quad L(\boldsymbol{v}) := \int_{\Omega} \boldsymbol{f}\cdot\boldsymbol{v}\, {\rm d}\Omega 
%\]
%that
%\[
%a(\boldsymbol{u},\boldsymbol{v}) -\int_{\Omega} p\, \nabla\cdot\boldsymbol{v}\, d\Omega = L(\boldsymbol{v})
%\]
%where $\boldsymbol{v}\in V$. To obtain an augmented Lagrangian formulation we write $\nabla\cdot\boldsymbol{v} = \nabla\cdot\boldsymbol{v} + \gamma q-\gamma q$
%for an arbitrary function $q\in Q$, so that we may write
%\[
%a(\boldsymbol{u},\boldsymbol{v}) +\int_{\Omega} p\, (\gamma q -\nabla\cdot\boldsymbol{v}) \, d\Omega -\int_{\Omega}\gamma p\, q \, d\Omega = L(\boldsymbol{v}) 
%\]
%Replacing the pressure in the first integral by the expression in (\ref{lambdadef}) we finally obtain the problem of finding $(\boldsymbol{u},p)\in V\times Q$ such that
%\begin{equation}\label{aug}
%a(\boldsymbol{u},\boldsymbol{v}) +\int_{\Omega} \frac{1}{\gamma}[\gamma p-\nabla\cdot\boldsymbol{u} ]_+ (\gamma q -\nabla\cdot\boldsymbol{v}) \, d\Omega -\int_{\Omega}\gamma p\, q \, d\Omega = L(\boldsymbol{v}) 
%\quad\forall (\boldsymbol{v},q)\in V\times Q \end{equation}
% This problem is related to seeking stationary points to the functional
we seek stationary points to the functional
\begin{equation}\label{augmin}
\mathcal{L}_\text{A}(\boldsymbol{u},p) := \frac12 a(\boldsymbol{u},\boldsymbol{u}) - L(\boldsymbol{u}) + \int_{\Omega}\frac{\gamma_0}{2}\left[\gamma_0^{-1} p-\nabla\cdot\boldsymbol{u}\right]_+^2d\Omega -\int_{\Omega}\frac{1}{2\gamma_0} p^2d\Omega
\end{equation}
analogously to (\ref{eq:compact_aug_final}).

For the discrete problem, we will use the inf--sup stable Taylor-Hood approximation which utilises the finite element space 
\[  
\vec{V}^h = \{\boldsymbol{v}:\boldsymbol{v} \in \left[ C^0(\Omega)\right]^d,\;
\boldsymbol{v}\vert_K\in [ P^2(K)]^d,  
\ \forall K\in {\mathcal T}^h,\;
\text{$\boldsymbol{v}=\boldsymbol{0}$ on $\Gamma^\text{D}$}\} 
\]  
for the velocity, where $P^2(K)$ denotes the space of piecewise quadratic polynomials on $K$, 
and the space 
$Q_h$ of piecewise linears for the pressure: 
\begin{equation} 
  \label{eq:polynomial-Lambda-p-1} 
  Q^{h} = \{p\in C^0(\Omega):\; p\vert_{K} 
\in P^{1}(K),\,\forall K\in {\mathcal T}^h\}.
\end{equation} 
The finite element method based on (\ref{augmin}) is to find $(\boldsymbol{u}^h,p^h)\in \vec{V}^h\times Q^{h}$  such that
\begin{equation}\label{aug1h}
a(\boldsymbol{u}^h,\boldsymbol{v}) -\int_{\Omega}{\gamma_0}[\gamma_0^{-1} p^h-\nabla\cdot\boldsymbol{u}^h ]_+ \nabla\cdot\boldsymbol{v}d\Omega = (\boldsymbol{f},\boldsymbol{v}) \quad \forall \boldsymbol{v}\in\vec{V}^h,
\end{equation}
and
\begin{equation}\label{aug2h}
\int_{\Omega}\left({\gamma_0}[\gamma_0^{-1} p^h-\nabla\cdot\boldsymbol{u}^h ]_+- p^h\right)qd\Omega =0  , \quad \forall q\in Q^{h} .
\end{equation}

\subsubsection{Satisfaction of assumptions for the abstract analysis}
For the present problem we have $V=[H^1(\Omega)]^n$, $H = H_h=H_h' = L^2(\Omega)$. The constraint operator $B$ is the divergence operator.
It if well known that the Taylor-Hood element admits a Fortin interpolant satisfying
\[
\|\nabla \pi_F \boldsymbol{v}\|_\Omega \lesssim \|\nabla \boldsymbol{v}\|_\Omega, \forall \boldsymbol{v} \in [H^1(\Omega)]^d.
\]
Since $\|\nabla \pi_F \boldsymbol{v}\|_\Omega \lesssim \|\nabla \pi_F \boldsymbol{v}\|_\Omega$ the relation \eqref{eq:Fortin_interp} holds. This means that for all $\mu_h \in Q_h$ there exists $\boldsymbol{v}_h \in \vec{V}^h $ such that $(\nabla \cdot \boldsymbol{v}_h, q_h)_\Omega = (\mu_h, q_h)_\Omega$ and $\|\boldsymbol{v}_h\|_V \lesssim \|\mu_h\|_\Omega$. Hence \eqref{eq:disc_infsup} is also satisfied.

Since $a(\cdot,\cdot)$ is a linear operator in this case we see that \eqref{eq:aprop1}-\eqref{eq:aprop3} are satisfied using standard arguments.
Hence the assumptions of section \ref{sec:AGLS} are satisfied in this case and hence
we conclude that the best approximation estimate
\eqref{eq:best_approx1} holds.

\subsection{Weak imposition of Dirichlet boundary conditions}

\subsubsection{Model problem}
Let us first consider the Poisson model problem: find $u:\Omega \rightarrow \IR$ such 
that
\begin{equation}\label{poiss}
-\Delta u= f ~\text{in} ~\Omega ,\quad u=g ~ \text{on} ~\Gamma:=\partial\Omega
\end{equation}
where $\Omega$ is a bounded domain in two or three space dimensions, with outward 
pointing normal $\boldsymbol{n}$, and $f$ and $g$ are given functions. For simplicity, 
we shall assume that $\Omega$ is polyhedral (polygonal).  
A classical way of prescribing $u=g$ on the boundary is to pose the problem (\ref{poiss}) as a minimisation problem with side conditions and seek stationary points to the functional
\begin{equation}\label{eq:first}
\mathcal{L}(v,\mu) := \frac12 a(v,v)  -
\lga{\mu}{v-g}  - \lom{f}{v} 
\end{equation}
where 
\begin{equation}\label{eq:forms}
\lom{f}{v} := \int_{\Omega}f v \, d\Omega, 
\quad  a(u,v) := \int_\Omega \nabla u\cdot\nabla v\, d\Omega
\end{equation}
and $\lga{\mu}{v-g}$ is interpreted as a duality pairing on $H^{-1/2}(\Gamma) \times H^{1/2}(\Gamma)$. We are thus in case B of Sec. \ref{sec:AGLS}, and the method proposed will only make sense on discrete spaces.

The stationary points to (\ref{eq:first}) are given by finding $(u,\lambda)\in H^1(\Omega)\times H^{-1/2}(\Gamma)$ such that
\begin{equation}\label{eq:weakform}
a(u,v) - \lga{\lambda}{v} = (f,v)\quad \forall v\in H^1(\Omega)
\end{equation}
\begin{equation}\label{eq:weakform2}
\lga{\mu}{u} = \lga{\mu}{g}\quad \forall \mu\in H^{-1/2}(\Gamma) 
\end{equation}
As mentioned above, the discretisation of this problem requires
balancing of the discrete spaces for the multiplier $\lambda$ and the primal solution $u$ in order for the method to be stable. 

\subsubsection{The augmented Lagrangian method for boundary conditions\label{Nitstab}}
%The other approach to deriving Nitsche's method from Lagrange multipliers is more in the vein of Nitsche' original paper \cite{Nit70} where the method was
%derived from a discrete minimization problem without multipliers. 
The Lagrangian in (\ref{eq:first}) is augmented by a penalty term scaled by a parameter $\gamma\in \mathbb{R}^+$ so that we seek stationary points to
\begin{equation}\label{eq:aug1}
\mathcal{L}_A(v,\mu) := \frac12 a(v,v)  -
\lga{\mu}{v-g}  + \frac{\gamma}{2}\| (v-g)\|^2_{H^{1/2}(\Gamma)} - \lom{f}{v} 
\end{equation}
\textcolor{black}{We note that the continuous norms imply $r=1/2$ in the discrete norms.
%The first derivative at $(v,\mu)$ is the linear form 
%\begin{equation}\label{eq:aug11}
%\mathcal{DL}_A|_{(v,\mu)}(\delta v,\delta \mu) = a(v,\delta v )  -
%\lga{\mu}{\delta v}  + \gamma(v-g,\delta v)_{H^{1/2}(\Gamma)} - \lom{f}{\delta v}  
%+ \lga{\delta \mu}{v-g} 
%\end{equation}
To find the stationary points we seek $(u,\lambda)$ such that }
%\begin{equation}
%\mathcal{DL}_A|_{(v,\mu)}(\delta v,\delta \mu)=0 \qquad \forall (\delta v, \delta \mu) \in H^1(\Omega) \times H^{1/2}(\Gamma)
%\end{equation}
%which corresponds to solving 
\begin{align*}
 a(u,v )  - \lga{\lambda}{v}  
+ \gamma(u-g, v)_{H^{1/2}(\Gamma)} 
+ \lga{u}{\mu} 
= {}&
\lom{f}{v}  \\ &
+ \gamma\langle g, v\rangle_{H^{1/2}(\Gamma)}\\ &
+ \lga{g}{\mu} 
\end{align*}
To determine the Lagrange multiplier $\lambda$ we set $\mu = 0$, and 
integrate by parts which gives
\begin{align}
(-\Delta v + f,  v )_\Omega + \langle\nabla_n u - \lambda, v\rangle_{H^{1/2}(\Gamma),H^{-1/2}(\Gamma)}  = 0 
\end{align}
For the exact solution the first term vanish and we conclude that $\lambda = \nabla_n v$.

We now wish to find a stable discrete counterpart to this optimisation problem.  To this end, let $\mathcal{T}_h$
be a family of quasi--uniform partitions, with 
mesh parameter $h$, of $\Omega$ into shape 
regular triangles or tetrahedra $T$ and the discrete space
\begin{equation}
V_h := \{v_h \in H^1(\Omega): v_h\vert_T \in \mathbb{P}_k(T), \,
\forall T \in {\mathcal{T}_h} \},\quad \mbox{ for } k \ge 1 
\end{equation}
and some discrete space $Q_h$ (not explicitly defined) for the approximation of the 
Lagrange multiplier.

We first follow the idea of (\ref{eq:discrete1}) and replace the $H^{1/2}$--norm by the discrete counterpart \textcolor{black}{$h^{-1/2} \| \cdot \|_{L_2(\Gamma)}$}, which by an inverse estimate dominates the $H^{1/2}(\Gamma)$ norm,
\begin{equation}\label{eq:inv}
\|v\|^2_{H^{1/2}(\Gamma)}\lesssim h^{-1} \|v\|^2_{L_2(\Gamma)} \quad \text{$v\in V_h$}
\end{equation}
and introduce the problem of finding the stationary point in $V_h\times Q_h$ of the discrete Lagrangian
\begin{equation}\label{eq:aug2a}
\mathcal{L}_A^h(v,\mu) := \frac12 a(v,v)  -
({\mu},{v-g})_\Gamma  +\frac{\gamma_0}{2h}\| v-g\|^2_{L_2(\Gamma)} - \lom{f}{v} 
\end{equation}

Recalling next that formally the Lagrange multiplier in (\ref{eq:weakform}) is given by $\mu = \nabla_n v$, which provides a direct way of computing the Lagrange multiplier from the primal solution, we obtain
\begin{equation}\label{eq:aug2}
\mathcal{L}_A^h(v) := \frac12 a(v,v)  -
(\nabla_n{v},{v-g})_{\Gamma}  +  \frac{\gamma_0}{2h}\| v-g \|^2_{L_2(\Gamma)}  - \lom{f}{v} 
\end{equation}
This is our stabilised ALM, the minimiser to which solves the problem of finding $u_h\in V_h$ such that
\textcolor{black}{
\begin{equation}\label{Nit1}
a(u_h,v)-({\nabla_n u_h},{v})_{\Gamma}-({\nabla_n v},{u_h})_{\Gamma}+\gamma_0 h^{-1} ({u_h},{v})_\Gamma = l(v)\quad\forall v\in V_h
\end{equation}
}
where
\begin{equation}
l(v) :=(f,v)+(\gamma_0 h^{-1}{v} - {\nabla_n v},{g})_\Gamma
\end{equation}
%\textcolor{black}{Doesn't the symmetry term appear naturally when deriving eq:aug2?}
%Note that we may symmetrize  (\ref{Nit1}) by subtracting  the consistent term
%$(u_h - g, \nabla_n v)_\Gamma$, which leads to 
We identify the classical method of Nitsche \cite{Nit70}, stable if $\gamma_0$ is chosen so that $\gamma _0> \gamma_C$, where $\gamma_C$ is the 
constant in the inverse inequality
\begin{equation}\label{eq:inverse}
h \|\nabla_n v\|_{L_2(\Gamma)}^2 \leq \gamma_C \|\nabla v\|_{L_2(\Omega)}^2
\end{equation}
\begin{remark}
As shown by Stenberg \cite{St95} (and discussed in Sec. \ref{sec:eliminate}), Nitsche's method can be viewed as a particular instance of the GLS stabilisation method of Barbosa--Hughes \cite{BaHu91}; in this sense the ALM is a variant of GLS, with the multiplier eliminated.
\end{remark}
\textcolor{black}{\begin{remark}
We note that the ALM leads to the symmetric form of Nitsche's method. The corresponding unsymmetric forms, as discussed, e.g., in \cite{ChHiRe15}, are derived using different arguments.
\end{remark}}

\subsection{Inequality boundary conditions}

An important feature of the augmented Lagrangian approach is that it can be extended 
to the case of inequality constraints, as first shown by Chouly and Hild in the context of elastic contact \cite{ChHi13}. We consider the problem: find $u:\Omega \rightarrow \IR$ such that 
\begin{equation}\label{poiss_ineq}
-\Delta u= f ~\text{in} ~\Omega ,\quad u-g\leq 0 ~ \text{on} ~\Gamma
\end{equation}
We have the following Kuhn--Tucker conditions on the multiplier and side condition:
\begin{equation}\label{eq:Kuhn1}
u -g \leq 0,\quad \lambda \leq 0, \quad \lambda (u-g)=0.
\end{equation}
We now use the analogue to (\ref{eq:lagsubst}), that
(\ref{eq:Kuhn1}) is equivalent to 
\begin{equation}\label{eq:lambda1}
\lambda = -{\gamma}\,[u-g-\gamma^{-1}\, \lambda]_+ 
\end{equation}
first used in this context by Alart and Curnier \cite{AC91}.
Now we can take another route to the augmented Lagrangian method. Taking the discrete counterpart to the standard multiplier equilibrium equation (\ref{eq:weakform}) we find
\begin{equation}\label{eq:weakformA}
(f,v) = a(u_h,v)-({\lambda_h},{v})_\Gamma = a(u_h,v)-({\lambda_h},{v-\gamma^{-1}\mu})_\Gamma-({\gamma^{-1}\lambda_h},{\mu})_\Gamma
\end{equation}
for all $v\in V_h$ and $\mu\in Q_h$ arbitrary.
Using now (\ref{eq:lambda1}) we find
\begin{equation}\label{eq:weakformB}
(f,v) = a(u_h,v)+({{\gamma}\,[u_h-g-\gamma^{-1}\, \lambda_h]_+},{v-\gamma^{-1}\mu})_\Gamma-({\gamma^{-1}\lambda_h},{\mu})_\Gamma\quad \forall (v,\mu)\in V_h\times Q_h.
\end{equation}
This is the optimality system for the Lagrangian
\begin{equation}\label{eq:aug_ineq_multi}
\mathcal{L}_A^h(v,\mu) := \frac12 a(v,v)  +\frac12 \Vert \gamma^{1/2}[v-g-\gamma^{-1} \mu]_+\Vert_{L_2(\Gamma)}^2 -\|\gamma^{-1/2}\mu\|^2_{L_2(\Gamma)} - \lom{f}{v} 
\end{equation}
\textcolor{black}{cf. \cite{AC91}.}
Approximating $\lambda_h\approx \partial_n u_h$ and setting $\mu=\partial_n v$, we seek $u_h\in V_h$ such that
 \begin{equation}\label{eq:weakformC}
a(u_h,v)+({{\gamma}\,[u_h-g-\gamma^{-1}\, \partial_nu_h]_+},{v-\gamma^{-1}\partial_n v})_\Gamma-({\gamma^{-1}\partial_n u_h},{\partial_n v})_\Gamma=(f,v)_\Omega \quad \forall \in V_h
\end{equation}
The solution to this problem is the minimiser of the nonlinear augmented Lagrangian
\begin{equation}\label{eq:augnonlin}
\mathcal{L}_A^h(v) := \frac12 a(v,v)  +
\frac12 \Vert \gamma^{1/2}[v_h-g-\gamma^{-1} \partial_nv]_+\Vert_{L_2(\Gamma)}^2 -\|\gamma^{-1/2}\partial_n v\|^2_{L_2(\Gamma)} - \lom{f}{v_h} 
\end{equation}

Again, we choose $\gamma = \gamma_0/h$.
% and note that (\ref{eq:aug2a}) can be rewritten
%\begin{equation}\label{eq:aug3}
%\mathcal{L}_A^h(v,\mu) := \frac12 a(v,v)  +
%\frac12 \Vert \gamma^{-1/2}(v-g-\gamma^{1/2} \mu)\Vert^2_{L_2(\Gamma)} -\Vert\gamma^{1/2}\,\mu \Vert^2_{L_2(\Gamma)} - \lom{f}{v} 
%\end{equation}
%which, with $\mu = \partial_nv$, is the linear counterpart to (\ref{eq:augnonlin}). Thus (\ref{eq:weakformC}) solves
%the inequality constrained problem approximately and is stable under the same condition (\ref{eq:inverse}) as its linear counterpart. 
Variants and several extensions of (\ref{eq:weakformC}) can be found in \cite{BuHa17}. We remark here that (\ref{eq:weakformC}) coincides with (\ref{Nit1}) in case of contact and gives a penalty on $\partial_n u = 0$  on $\Gamma$ in case of no contact. This penalty does not destroy the coercivity of the problem if (\ref{eq:inv}) is satisfied.

\begin{remark}
In the GLS stabilisation for variational inequalities proposed by Barbosa and Hughes \cite{BH92}, no penalty is added to the Lagrangian;
the multiplier is not eliminated, and their approach is a stabilised Lagrange multiplier method which requires the solution of 
an inequality problem. It is also possible to retain the multiplier in the ALM and add GLS stabilisation to the augmented Lagrangian. 
This approach, which also leads to a nonlinear equality problem, was explored in \cite{HaRaSa16}.
\end{remark}
\subsubsection{Satisfaction of assumptions for the abstract analysis}
In this case $V=H^1(\Omega)$ and $H=H^{\frac12}(\partial \Omega)$, $H' = H^{-\frac12}(\partial \Omega)$. However since the solution to \eqref{poiss_ineq} is known to have the additional regularity $u \in H^{\frac32+\epsilon}(\Omega)$, $\epsilon>0$ it follows that $\partial_n u \in L^2(\Omega)$ and the discrete norms $H_h$ and $H_h'$ defined by \eqref{eq:discrete1} and \eqref{eq:discrete2} are well defined on the exact solution.
While \eqref{eq:inverse} then is enough to make the formulation \eqref{eq:weakformC} satisfy the assumptions necessary for the analysis of section \ref{sec:eliminate}, the formulation \eqref{eq:aug_ineq_multi} still requires the satisfaction of \eqref{eq:Fortin_interp} and \eqref{eq:disc_infsup}. For a charaterisation of spaces satisfying these conditions (in the $h$-weighted $L^2$-norm) we refer to \cite{Pit80}. An example of a construction is two space dimension is to take element wise constant approximation for $Q_h$ and let $V_h$ consist of piecewise quadratic continuous approximation, or piecewise affine approximation enriched with a quadratic bubble added to elements adjacent to the boundary on each boundary face. The Fortin interpolant can then be constructed by first defining the nodal degrees of freedom using any $H^1$-stable interpolant and then fixing the degree of freedom associated to the bubble on each boundary faces so that \eqref{eq:Fortin_interp} and \eqref{eq:disc_infsup} are satisfied. Indeed here they are equivalent. The same construction may be used for the forthcoming sections.  %% Should we say more? Formulas-proofs?
\subsection{A model for elastic contact\label{contstab}}
\subsubsection{Treatment of Robin boundary conditions}
To show the versatility of the ALM 
we shall consider the equations of linear elasticity in contact with a springy substrate. We start with the linear case of a Robin boundary condition: Find the displacement $\bfu = \left[
u_i\right]_{i=1}^n$ and the symmetric stress tensor $\bfsig =
\left[\sigma_{ij}\right]_{i,j=1}^n$ such that
\begin{align}\label{diffelastis}
  \bfsig = {} & \frac{\nu E}{(1+\nu)(1-2\nu)} ~\text{tr}\,\bfeps(\bfu)\,\bfI 
   + \frac{E}{(1+\nu)}\bfeps(\bfu)\quad \text{in}\quad\Omega ,\\ 
-\nabla\cdot\bfsig = {}& \bff \quad \text{in}\quad\Omega, \\
  \bfS\bfu ={}& -\bfsig\cdot\bfn\quad \text{on}\quad\partial\Omega_{\text{S}}, \\
\bfsig\cdot\bfn = {}& {\bf 0}  \quad \text{on}\quad\partial\Omega\setminus\partial\Omega_{\text{S}}.\label{diffelastie}
\end{align}
Here $\Omega$ is a closed subset of $\IR^n$, $n=2$ or $n=3$, 
%\[
%\lambda=\frac{\nu E}{(1+\nu)(1-2\nu)}\quad \text{and}\quad\mu=\frac{E}{2(1+\nu)}
%\] 
$E$ is Young's modulus and $\nu$ is Poisson's ratio.
%Lam\'{e} parameters, and
$\bfeps\left(\bfu\right) = \left[\varepsilon_{ij}(\bfu)\right]_{i,j=1}^n$
is the strain tensor with components
\[ \varepsilon_{ij}(\bfu ) = \frac{1}{2}\left( \frac{\partial
  u_i}{\partial x_j}+\frac{\partial u_j}{\partial x_i}\right),
  \]
and trace
  \[
\text{tr}\,\bfeps(\bfu) = \sum_{i} \varepsilon_{ii}(\bfu) = \nabla\cdot\bfu .
\] 
Furthermore,
$\nabla\cdot\bfsig = \left[\sum_{j=1}^n\partial
  \sigma_{ij}/\partial x_j\right]_{i=1}^n$,
$\bfI = \left[\delta_{ij}\right]_{i,j=1}^n$ with $\delta_{ij} =1$
if $i=j$ and $\delta_{ij}= 0$ if $i\neq j$, and $\bff$ is a
given load. Finally, we assume that the boundary stiffness $\bfS$ is of the form
\[
\bfS = \alpha^{-1} \bfn\otimes\bfn + \beta^{-1} \bfP, \quad \bfP:=(\bfI-\bfn\otimes\bfn)
\]
where $\alpha$ and $\beta$ are flexibility parameters in the normal and tangential direction, respectively.
The solution to (\ref{diffelastis})--(\ref{diffelastie}) minimises the functional
\begin{equation}\label{eq:stifform}
\mathcal{L}_S(\bfu) := \frac{1}{2}a(\bfu,\bfu) -(\bff,\bfu)_{\Omega} + \langle\bfS \bfu,\bfu\rangle_{\partial\Omega_\text{S}}
\end{equation}
where
\[
a(\bfu,\bfv) := ( \bfsig(\bfu),\bfeps(\bfv))_{\Omega}=\int_{\Omega} \bfsig(\bfu):\bfeps(\bfv)\, d\Omega
\]
which is the usual foundation for a discrete method. However, to obtain a robust method for the case of $\alpha\rightarrow 0$ or $\beta\rightarrow 0$, we can introduce a new variable $\bflam\in[ L_2(\partial\Omega_\text{S})]^n$ and seek stationary points to
\begin{equation}\label{eq:flex}
\mathcal{L}(\bfu,\bflam) := \frac{1}{2}a(\bfu,\bfu) -(\bff,\bfu)_{\Omega} -\frac{1}{2} \langle\bfK\bflam  ,\bflam\rangle_{\partial\Omega_\text{S}}- \langle\bflam  ,\bfu\rangle_{\partial\Omega_\text{S}}
\end{equation}
where $\bfK := \bfS^{-1}$ is a flexibility matrix which simply tends to the zero matrix if $\alpha,\beta\rightarrow 0$, and the Robin condition becomes a Dirichlet condition.
The stationary point to (\ref{eq:flex}) fulfils the variational equations of finding $(\bfu,\bflam)\in [H^1(\Omega)]^n\times [L_2(\partial\Omega_\text{S})]^n$ such that
\begin{align}
a(\bfu,\bfv) - \langle\bflam  ,\bfv\rangle_{\partial\Omega_\text{S}} = {}& (\bff,\bfv)_{\Omega}\quad\forall \bfv\in [H^1(\Omega)]^n\\
 \langle\bfK\bflam  +\bfu,\bfmu\rangle_{\partial\Omega_\text{S}}= {}& { 0}\quad\forall \bfmu\in [L_2(\partial\Omega_\text{S})]^n
\end{align}
and we note that, formally, 
\begin{equation}\label{eq:bflambda}
\bflam = \bfsig(\bfu)\cdot\bfn
\end{equation}
In the discrete case, we can now formulate an ALM by adding a penalty term and replacing $\bflam$ using (\ref{eq:bflambda}), looking for the minimiser of
\begin{align}\nonumber
\mathcal{L}_A^h(\bfu) :={}& \frac{1}{2}a(\bfu,\bfu) -(\bff,\bfu)_{\Omega} - \langle\bfsig(\bfu)\cdot\bfn  ,\bfu\rangle_{\partial\Omega_\text{S}}\\
 & -\frac12 \langle\bfK\bfsig(\bfu)\cdot\bfn  ,\bfsig(\bfu)\cdot\bfn\rangle_{\partial\Omega_\text{S}}+\frac12\langle\bfS_h (\bfK\bfsig(\bfu)\cdot\bfn+\bfu),\bfK\bfsig(\bfu)\cdot\bfn+\bfu\rangle_{\partial\Omega_\text{S}}
\label{eq:flexALM}\end{align}
where $\bfS_h$ is a discrete stiffness matrix, to be chosen. The minimiser to (\ref{eq:flexALM}) satisfies the variational equation of finding $\bfu_h\in [V_h]^n$ such that
\begin{equation}\label{eq:varsh}
a_{\bfS_h}(\bfu_h,\bfv ) = (\bff,\bfv)_{\Omega}\quad\forall \bfv\in V := [V_h]^n
\end{equation}
where
\begin{align}\label{eq:varshb}
a_{\bfS_h}(\bfu,\bfv ) := {}& a(\bfu,\bfv) - \langle\bfu +\bfK\bfsig(\bfu)\cdot\bfn,\bfsig(\bfv)\cdot\bfn\rangle_{\partial\Omega_\text{S}}- \langle\bfsig(\bfu)\cdot\bfn,\bfv +\bfK\bfsig(\bfv)\cdot\bfn\rangle_{\partial\Omega_\text{S}}\\
& + \langle \bfK\bfsig(\bfu)\cdot\bfn,\bfsig(\bfv)\cdot\bfn\rangle_{\partial\Omega_\text{S}}+ \langle  \bfS_h(\bfu+\bfK\bfsig(\bfu)\cdot\bfn),\bfv+\bfK\bfsig(\bfv)\cdot\bfn\rangle_{\partial\Omega_\text{S}}
\end{align}
which is related to the Nitsche method for interfaces in \cite{HaHa04,Ha05}, and a variant of the method of Juntunen and Stenberg \cite{JuSt09} for Poisson's problem with Robin boundary conditions.
 With the particular choice 
\begin{equation}
\bfS_h = \left((h/\gamma_0)\bfI + \bfK\right)^{-1}
\end{equation}
we regain the standard Nitsche method for the Dirichet problem if $\bfK$ is the zero matrix, and if $\bfK$ is nonzero we approach the minimiser of (\ref{eq:stifform}) as $h\rightarrow 0$.
Thus the method is robust also in the limit of zero flexibility.

\subsubsection{One--sided conditions in contact}\label{sec:one-sided}
We now wish to activate the Robin boundary only if $\bfu\cdot\bfn -g > 0$, corresponding to contact with a springy foundation at a distance $g$ from the elastic body. Since this condition is only on the normal part of the displacement, we consider the case of slip, i.e., we choose
\[
\bfK = \alpha \bfn\otimes\bfn
\]
Setting $\sigma_n := \bfn\cdot\bfsig\cdot\bfn$ and $u_n =\bfu\cdot\bfn$, the linear case is then to find stationary points to (\ref{eq:flex}) simplified as
%\begin{align}\nonumber
%\mathcal{L}_A^h(\bfu) :={}& \frac{1}{2}a(\bfu,\bfu) -(\bff,\bfu)_{\Omega} - \langle\bfsig_n(\bfu) ,u_n\rangle_{\partial\Omega_\text{S}}\\
% & -\frac12 \langle\alpha\sigma_n(\bfu)  ,\sigma_n(\bfu)\rangle_{\partial\Omega_\text{S}}+\frac12\langle s_h (\alpha\bfsig_n(\bfu)+u_n),\alpha\bfsig_n(\bfu)+u_n\rangle_{\partial\Omega_\text{S}}, \quad 
%\label{eq:flexALM2}\end{align}
\begin{equation}\label{eq:simpmod}
\mathcal{L}(\bfu,\lambda_n) := \frac{1}{2}a(\bfu,\bfu) -(\bff,\bfu)_{\Omega} -\frac{1}{2} \langle\alpha\lambda_n   ,\lambda_n \rangle_{\partial\Omega_\text{S}}- \langle\lambda_n  ,u_n-g\rangle_{\partial\Omega_\text{S}}
\end{equation}
%where
%\begin{equation}
%s_h := \frac{\gamma}{h + \gamma\alpha}
%\end{equation}
where formally $\lambda_n = \sigma_n(\bfu)$. In the case of contact we now have the KKT condition
\begin{align}
 u_n -g+\alpha\lambda_n \leq {}& 0\\
\lambda_n \leq {}& 0\\
\lambda_n \left( u_n -g+\alpha\lambda_n\right) = {}& 0
\end{align} 
which we can formally rewrite as
\begin{equation}\label{eq:KKTcontact}
\lambda_n =-\gamma [  ( u_n -g+\alpha\lambda_n)-\gamma^{-1}\lambda_n ]_+
\end{equation}
%or 
%\begin{equation}\label{eq:KKTcontact}
%\sigma_n (\bfu)=-\gamma [  ( u_n -g+(\alpha-\gamma^{-1})\sigma_n (\bfu) ]_+
%\end{equation}
Proceeding as in (\ref{eq:weakformA}), the equilibrium equation resulting from (\ref{eq:simpmod}) is 
\begin{equation}
(\bff,\bfv)_{\Omega} = a(\bfu,\bfv) -\langle\lambda_n  ,v_n\rangle_{\partial\Omega_\text{S}}
\end{equation}
and seeing as 
\begin{align}\nonumber
- \langle\lambda_n  ,v_n\rangle_{\partial\Omega_\text{S}} = {}& - \langle\lambda_n  ,v_n+\alpha\mu_n\rangle_{\partial\Omega_\text{S}} +\langle\alpha \lambda_n ,\mu_n\rangle_{\partial\Omega_\text{S}} \\
 = {}& - \langle\lambda_n  ,v_n+(\alpha-\gamma^{-1})\mu_n\rangle_{\partial\Omega_\text{S}} +\langle(\alpha-\gamma^{-1}) \lambda_n, \mu_n\rangle_{\partial\Omega_\text{S}} 
  \end{align}
  with $\mu_n$ arbitrary, we find that the discrete augmented Lagrangian can be written
\begin{align}\nonumber
\mathcal{L}_A^h(\bfu,\lambda_n) :={}& \frac{1}{2}a(\bfu,\bfu) -(\bff,\bfu)_{\Omega} +\frac12\| \gamma^{1/2}[(u_n-g + (\alpha-\gamma^{-1}) \lambda_n )]_+\|^2_{\partial\Omega_\text{S}}
\\
 & +\frac12 \langle(\alpha-\gamma^{-1})\lambda_n,\lambda_n\rangle_{\partial\Omega_\text{S}}
 \end{align}
%
%the Lagrangian (\ref{eq:flexALM}), here simplified to
%\begin{align}\nonumber
%\mathcal{L}_A^h(\bfu) :={}& \frac{1}{2}a(\bfu,\bfu) -(\bff,\bfu)_{\Omega} - \langle\sigma_n(\bfu)  ,u_n-g\rangle_{\partial\Omega_\text{S}}\\
% & -\frac12 \langle\alpha\sigma_n(\bfu),\sigma_n(\bfu)\rangle_{\partial\Omega_\text{S}}+\frac12\langle\gamma (\alpha\sigma_n(\bfu)+u_n-g),\alpha\sigma_n(\bfu)+u_n\rangle_{\partial\Omega_\text{S}}
%\end{align}
%%with $\gamma = (h/\gamma_0+\alpha)^{-1}$, 
%can be rewritten
%\begin{align}\nonumber
%\mathcal{L}_A^h(\bfu) :={}& \frac{1}{2}a(\bfu,\bfu) -(\bff,\bfu)_{\Omega} +
%\frac12\| \gamma^{1/2}(u_n-g + (\alpha-\gamma^{-1}) \sigma_n(\bfu) )\|^2_{\partial\Omega_\text{S}}\\
% & +\frac12 \langle(\alpha-\gamma^{-1})\sigma_n(\bfu)  ,\sigma_n(\bfu)\rangle_{\partial\Omega_\text{S}}
%\end{align}
%Invoking (\ref{eq:KKTcontact}) gives the nonlinear discrete augmented Lagrangian 
and  with $\lambda_n \approx \sigma_n(\bfu)$,
\begin{align}\nonumber
\mathcal{L}_A^h(\bfu) :={}& \frac{1}{2}a(\bfu,\bfu) -(\bff,\bfu)_{\Omega} +
\frac12\| \gamma^{1/2}[u_n-g + (\alpha -\gamma^{-1})\sigma_n(\bfu)]_+ \|^2_{\partial\Omega_\text{S}}\\
 & +\frac12 \langle(\alpha-\gamma^{-1})\sigma_n(\bfu)  ,\sigma_n(\bfu)\rangle_{\partial\Omega_\text{S}}
\label{eq:flexALM3}\end{align}
the minimiser of which is $\bfu\in V$ satisfying
\begin{align}\nonumber
(\bff,\bfv)_{\Omega} ={}& a(\bfu,\bfv) + \langle\gamma [u_n-g + (\alpha -\gamma^{-1}) \sigma_n(\bfu)]_+,v_n + (\alpha -\gamma^{-1}) \sigma_n(\bfv)\rangle \\
 & +\langle(\alpha-\gamma^{-1})\sigma_n(\bfu)  ,\sigma_n(\bfv)\rangle_{\partial\Omega_\text{S}} \quad\forall \bfv\in V
\label{eq:flexsolve}\end{align}
which coincides with (\ref{eq:varsh}) in contact, and gives an additional penalty on the condition $\sigma_n(\bfu) =0$ if there is no contact. 
Choosing now 
\begin{equation}\label{eq:alphapen}
\gamma = (h/\gamma_0+\alpha)^{-1}  \; \Rightarrow \; \alpha-\gamma^{-1} = -\frac{h}{\gamma_0}
\end{equation}
we obtain
the same penalty on the normal stress as in \cite{ChHi13}, which does not destroy the positive definite nature of the problem if we take
$\gamma_0 > \gamma_C$ where $\gamma_C$ is the (stiffness dependent) constant in the inverse inequality
\begin{equation}\label{eq:inverse2}
\| h^{1/2}\bfsig(\bfv)\cdot\bfn\|_{\partial\Omega_\text{S}}^2 \leq \gamma_C a(\bfv,\bfv)\, \quad\forall \bfv \in V
\end{equation}

\subsection{Stabilising the Kirchhoff plate model}

\subsubsection{Approximation with independent rotations and displacement}

In the Kirchhoff plate model, posed on a domain 
$\Omega\subset \IR^2$ with boundary $\partial\Omega$, we seek an out--of--plane 
(scalar) displacement $u$ to which we associate the strain (curvature) tensor
\begin{equation}
\bfeps(\nabla u) := \frac12\left(\nabla\otimes (\nabla u) + (\nabla u ) \otimes \nabla \right) 
= \nabla \otimes \nabla u = \nabla^2 u   
\end{equation}
and the plate stress (moment) tensor
\begin{align}\label{eq:plate-stress-tensor}
\bfsig_P (\nabla u) &:= \mcCP \left(\bfeps(\nabla u) + \nu (1- {\nu })^{-1}
\nabla\cdot\nabla u \, \bfI \right)
\\
&= \mcCP \left( \nabla^2 u + \nu (1-\nu)^{-1} \Delta u \bfI \right)
\end{align}
where 
\begin{equation}\label{eq:mcCP}
\mcCP =  \frac{E t^3}{12(1+\nu)} 
\end{equation}
where $t$ denotes the 
plate thickness. 

The Kirch\-hoff clamped problem then takes the form: given the out--of--plane 
(scaled) load $t^3f$, find the displacement $u$ such that
\begin{align}
\nabla\cdot\left( \nabla\cdot \bfsig_P ( \nabla u )\right)= t^3f &  \qquad \text{in $\Omega$}
\\
u = 0  & \qquad \text{on $\partial \Omega$}
\\
\bfn \cdot \nabla u = 0   &\qquad \text{on $\partial \Omega$}
\end{align}
The corresponding variational problem takes the form: 
Find the displacement $u \in H^2_0(\Omega)$ such that
\begin{equation}\label{eq:varform}
a_P(\nabla u, \nabla v )= (f,v)_\Omega \qquad \forall v \in H^2_0(\Omega
\end{equation}
where 
\begin{align}
a_P(\nabla v, \nabla w) & := (t^{-3}\bfsig_P(\nabla v), 
\bfeps(\nabla w))_\Omega 
\end{align}

From a computational point of view (\ref{eq:varform}) is cumbersome since it requires $C^1$--conforming elements or
carefully constructed nonconforming approximations. It is therefore common to use instead the Mindlin--Reissner model
which
is described by the following partial differential equations:
\begin{equation}
\label{PDE}
\begin{split}
  -t^{-3}\nabla\cdot\bfsig_P (\bftheta)-\kappa\,t^{-2}\left(\nabla u-\bftheta\right) &= 0, \quad \text{in } \Omega\subset\IR^2, \\
  -\kappa\,t^{-2}\,\nabla\cdot\left(\nabla u-\bftheta\right) &= f, \quad \text{in } \Omega,
\end{split}
%\begin{array}{>{\displaystyle}l}
%-\nabla\cdot\bfsig (\bftheta)-\kappa\,t^{-2}\left(\nabla u-\bftheta\right)=0\quad\text{in}\quad\Omega\subset\IR^2 ,\\[4mm]
%-\kappa\,t^{-2}\,\nabla\cdot\left(\nabla u-\bftheta\right)=g\quad\text{in}\quad\Omega ,
%\end{array}
\end{equation}
where
$\bftheta$ is the
rotation of the median surface and
$\kappa$ is a shear correction factor.
We note that this relaxes the continuity requirement on $u$ and that, as $t\rightarrow 0$, tends to the Kirchhoff model.
However, the requirement on the approximation to allow $\vert\nabla u -\bftheta\vert \rightarrow 0$ is difficult to realise
in the discrete setting and if this condition cannot be met, shear locking occurs, destroying the approximation properties of the discrete model.
The ALM can offer an alternative approach in which we enforce the requirement $\nabla u = \bftheta$ by a Lagrange multiplier.
To this end we consider the Lagrangian
\begin{equation}\label{eq:MRfunc}
 \mathcal{L}(u,\bftheta,\bfp) := \frac12 a_P(\bftheta ,\bftheta)+(\bfp,\nabla u-\bftheta)_{\Omega}-(f,u)_{\Omega}
\end{equation}
The Euler stationary points of (\ref{eq:MRfunc}) satisfy the weak system
\begin{align}
a_P(\bftheta,\bftwei) +(\bfp,\nabla v-\bftwei)_{\Omega} = {}& (f, v)_{\Omega} \quad\forall (v,\bftwei)\in  H_0^1(\Omega)\times [H_0^1(\Omega)]^2,\\
(\nabla u-\bftheta,\bfq)_{\Omega} = {}& 0\quad \forall \bfq\in [L_2(\Omega)]^2,
 \end{align}
 corresponding to the strong form
 \begin{align}
  -\nabla\cdot\bfsig_P (\bftheta) &= t^3\bfp \quad \text{in } \Omega\label{eq:lamsig} \\
  -\nabla\cdot\bfp &= f \quad \text{in } \Omega\\
  \nabla u-\bftheta &= 0 \quad \text{in } \Omega
\end{align}

We now wish to stabilise (\ref{eq:MRfunc}) using the ALM. To this end, we
use (\ref{eq:lamsig}) to eliminate $\bflam$ and add a penalty term on the side condition to obtain the augmented discrete functional
\begin{equation}
 \mathcal{L}_A^h(u_h,\bftheta_h):= \frac12 a_P(\bftheta_h,\bftheta_h) -(t^{-3}\nabla\cdot\bfsig_P (\bftheta_h),\nabla u_h-\bftheta_h)_{h} + \frac{\gamma}{2}\| \nabla u_h -\bftheta_h\|^2_{\Omega}-(f,u_h)_{\Omega}
\end{equation}
where $u_h\in V^h_1$ and $\bftheta_h\in [V^h_2]^2$ for some discrete spaces $V^h_1$ and $V^h_2$. Here we use the notation
\begin{equation}
(\bfu,\bfv)_h := \sum_{T\in\mathcal{T}_h} \int_{T} \bfu\cdot\bfv\, dxdy
\end{equation}

The Euler equations corresponding to the augmented system are
\begin{equation} \label{pdef0}
A_h((\bftheta_h,u_h),(\bftwei,v))=  (f, v)\
 \end{equation}
for all $(v,\bftwei)\in V^h_1\times [V^h_2]^2$, where
\begin{equation}
A_h((\bftheta,u),(\bftwei,v)) := a_P(\bftheta,\bftwei) -(t^{-3}\nabla\cdot\bfsig_P (\bftheta),\nabla v-\bftwei)_{h}-(\nabla u-\bftheta,t^{-3}\nabla\cdot\bfsig_P (\bftwei))_{h}+\gamma (\nabla u -\bftheta,\nabla v-\bftwei)
\end{equation}
 Now, if $V_2^h$ is the space of piecewise linears, the terms $(\cdot,\cdot)_h$ vanish
and, seeing as $\theta\in H^{1}(\Omega)$ and thus $\lambda\in H^{-1}(\Omega)$, we choose \textcolor{black}{$r=1$} in (\ref{eq:discrete2}) and $\gamma = \gamma_0/h^2$ to obtain a scheme proposed by Pitk\"{a}ranta \cite{Pi88}; for higher order polynomial approximations we recover
a GLS stabilisation method due to Stenberg \cite{St95b,BeNiSt07}.

\subsubsection{The plate obstacle problem}

We next consider applying the model from the previous Section to a regularised plate obstacle problem.  The continuous model is
 \begin{align}
  -\nabla\cdot\bfsig_P (\bftheta) &= t^3\bfp \quad \text{in } \Omega\label{eq:obst} \\
  -\nabla\cdot\bfp + p &= f \quad \text{in } \Omega\\
  \nabla u-\bftheta &= 0 \quad \text{in } \Omega\\
  p\geq 0,\;  u-g + \beta p\geq 0, \; p(u-g+\beta p) &= 0 \quad \text{in } \Omega
  \end{align}
 Here, $\beta$ is a given compliance which regularises the problem, in the limit case of $\beta=0$
 (rigid obstacle) we instead have the KKT conditions $ p\geq 0$, $u-g\geq 0$, and  $p(u-g)= 0$.  Note  that the regularity in the limit case is insufficient for the analysis above. 
 Indeed it is well known that $u \not \in H^4(\Omega)$, which is insufficient for the multiplier to be in $L_2$.
 It is however known that for $\beta>0$, $u \in H^4(\Omega)$ if the interior angles of the domain are smaller than $126^\circ$ (see \cite{BR80}). Therefore the analysis is valid for all $\beta >0$,
 since we have 
 \begin{equation}
p\in Q = \left\{\begin{array}{>{\displaystyle}l}
 L_2(\Omega)\quad\text{if $\beta > 0$} \\
 H^{-2}(\Omega)\quad\text{if $\beta = 0$}\end{array}\right.
 \end{equation} 

We see that, again, formally $\bflam = -t^{-3}\nabla\cdot\bfsig_P (\bftheta)$ and that $p = f -t^{-3}\nabla\cdot\left( \nabla\cdot \bfsig_P (\bftheta )\right)$. Following the
strategy from Sec. \ref{sec:one-sided} we write
\begin{equation}
p =\epsilon  [  ( u_n -g+\beta p)-\epsilon^{-1}p ]_+
\end{equation}
We need to also stabilise the rotations, and
to this end we consider the discrete Lagrangian
\begin{align}\nonumber
\mathcal{L}_A^h(\bftwei,v) := &{} \frac12 A_h((\bftwei,v),(\bftwei,v))   +
\frac12 \Vert \epsilon^{1/2}[u-g-(\epsilon^{-1}-\beta) (f -t^{-3}\nabla\cdot\left( \nabla\cdot \bfsig_P (\bftwei) )\right)]_+\Vert_h^2 \\
& -\frac12\|(\epsilon^{-1}-\beta)^{1/2} (f -t^{-3}\nabla\cdot ( \nabla\cdot \bfsig_P (\bftwei) ) )\|_h^2 - \lom{f}{v} \label{eq:augnonlinplate}
\end{align}
where, considering the limit case $p\in H^{-2}(\Omega)$, we choose $r=2$ and thus 
\begin{equation}
\epsilon = (h^4/\gamma_1+\beta)^{-1}
\end{equation}
 with $\gamma_1$ a sufficiently large constant. A similar approach has been suggested by Gustafsson et al. \cite{GuStVi17,GSV19} in the context of $C^1$ approximations of the clamped Kirchhoff plate with GLS stabilisation, without specific reference to augmented Lagrangian methods.

\section{Numerical examples}\label{sec:numerics}
\subsection{Cavitation}
\textcolor{black}{The problem formulation is that of (\ref{aug1h})--(\ref{aug2h}). Our numerical experience is that for the chosen discretization $\gamma_0$ should not be chosen too large; in our example we chose $\gamma_0=1/100$.}

We consider a domain with an elliptically shaped pocket, with mesh shown in Fig. \ref{fig:meshcav}. The boundary conditions are natural boundary conditions $(-p\boldsymbol I + \mu\nabla \boldsymbol u)\cdot\boldsymbol n = \boldsymbol 0$ at the left- and right-hand sides. The velocity is set to zero along the floor of the channel and pocket boundary, and the flow is driven by setting $\boldsymbol u = (1,0)$ at the ceiling. The viscosity is $\mu=1$.
We compare the pressure solution with and without cavitation in Figs. \ref{fig:caviso}--\ref{fig:cavelev} and
note that there is a pressure resultant in the cavitation case, creating a lifting resultant force, cf. \cite{NiHa11}. 

\subsection{Elastic contact with flexible plane}

In this example, we consider an elastic sphere of radius 1 under the load $\bff=(0,0,-50)$ in contact with a flexible plane.
\textcolor{black}{The contact is assumed friction--free, in accordance with the form (\ref{eq:flexsolve}).}
The moduli of elasticity were chosen as $E=200$ and $\nu=0.33$ and the stabilisation parameter was taken as $\gamma=100 E$.
In Figs. \ref{fig:alpha0}--\ref{fig:alphaminus2} we show the deformation
and contact pressure for increasing flexibilities of the contact plane.

\subsection{Plate obstacle problem}
The considered example, from \cite{GSV19}, concerns a clamped square plate $\Omega = (0,1)\times(0,1)$ in contact with a rigid obstacle ($\beta=0)$ in the center of the plate,
$g=100((x-1/2)^2+(y-1/2)^2)$. Here $E=1$, $\nu=0$, $t=1$, and we chose $\gamma_1=10 E$ and $\gamma_2 =E/10$. We present a sample computation using continuous, piecewise $P^2$ approximations for both displacement and rotations on triangular meshes, \textcolor{black}{based on the variational equations resulting from minimization of the Lagrangian (\ref{eq:augnonlinplate}).}
The mesh is shown in Fig. \ref{fig:platemesh} (left), and the corresponding soultion is given
in Figs. \ref{fig:platemesh} (right, with obstacle indicated) and \ref{fig:plateiso}. The computational solution agrees well with that of \cite{GSV19}.

%\bibliographystyle{abbrv}
%\bibliography{AUG}
%\input{referenc}
\newpage
\begin{figure}[hbt]
\begin{center}
\includegraphics[width=3in]{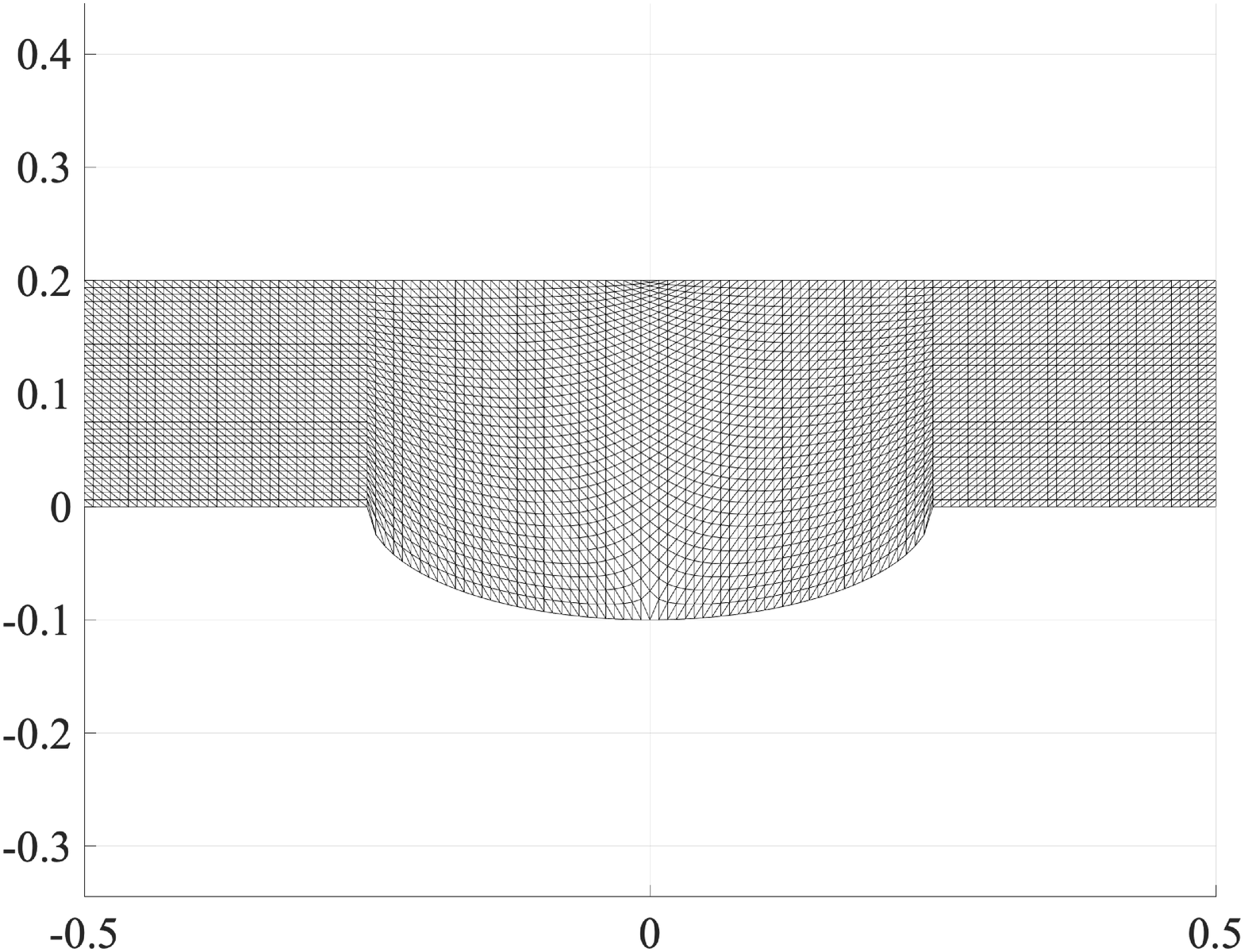}
\end{center}
\caption{Mesh for cavitation computations.\label{fig:meshcav}}
\end{figure}
\begin{figure}[hbt]
\begin{center}
\includegraphics[width=3in]{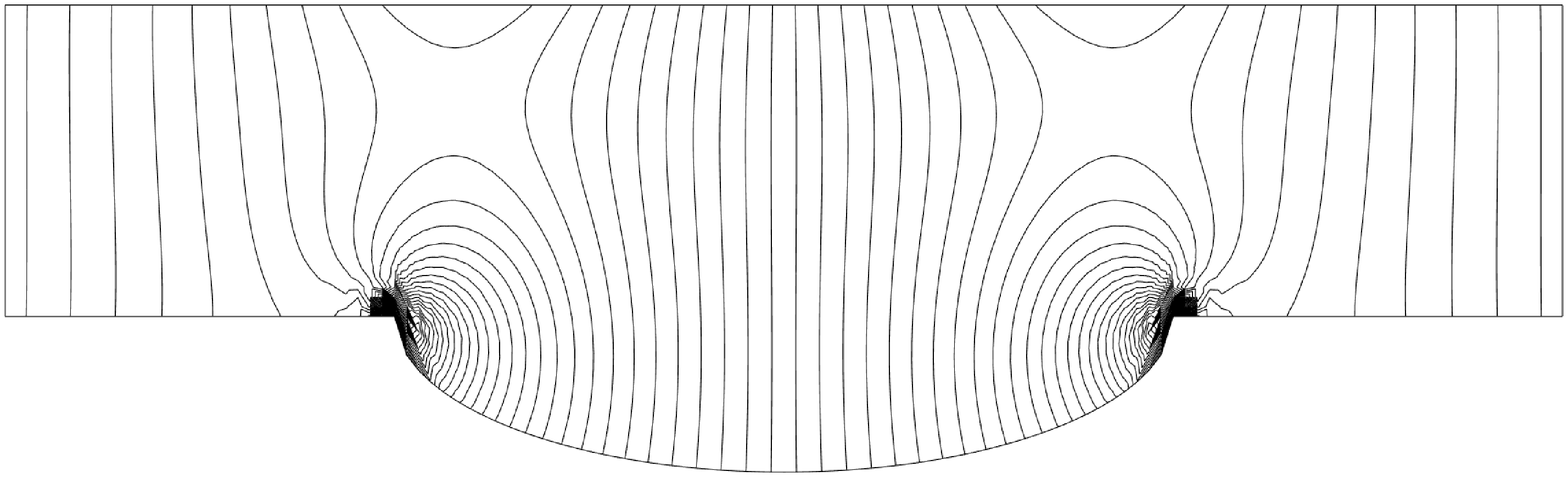}\includegraphics[width=3in]{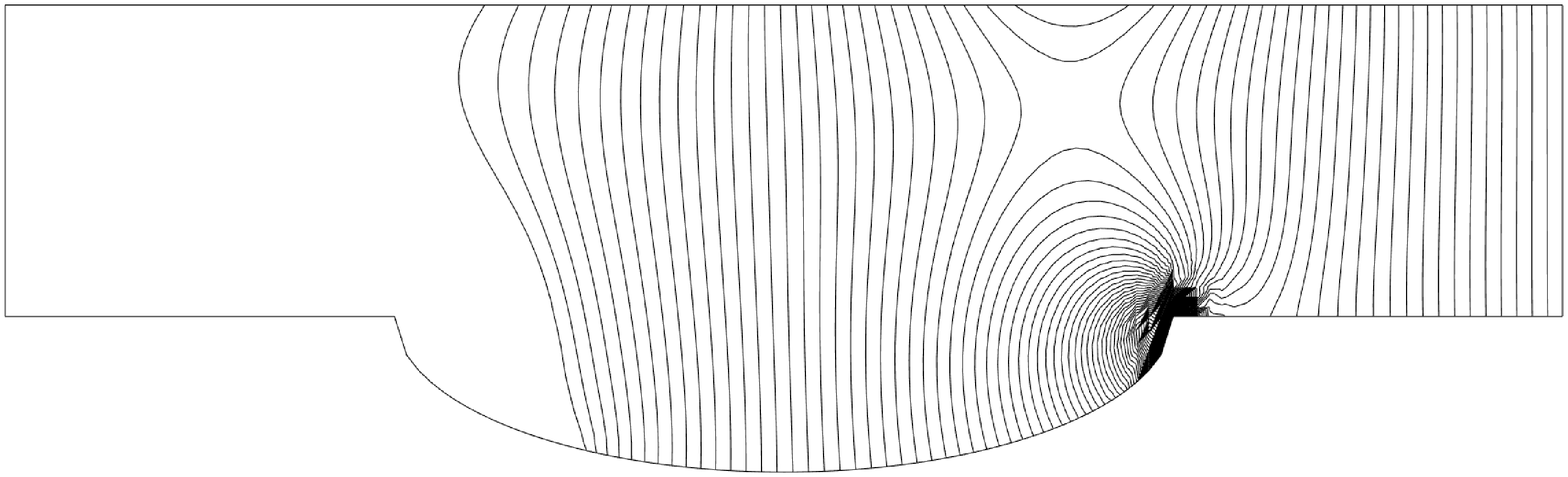}
\end{center}
\caption{Pressure isolines without (left) and with (right) cavitation.\label{fig:caviso}}
\end{figure}
\begin{figure}[hbt]
\begin{center}
\includegraphics[width=3in]{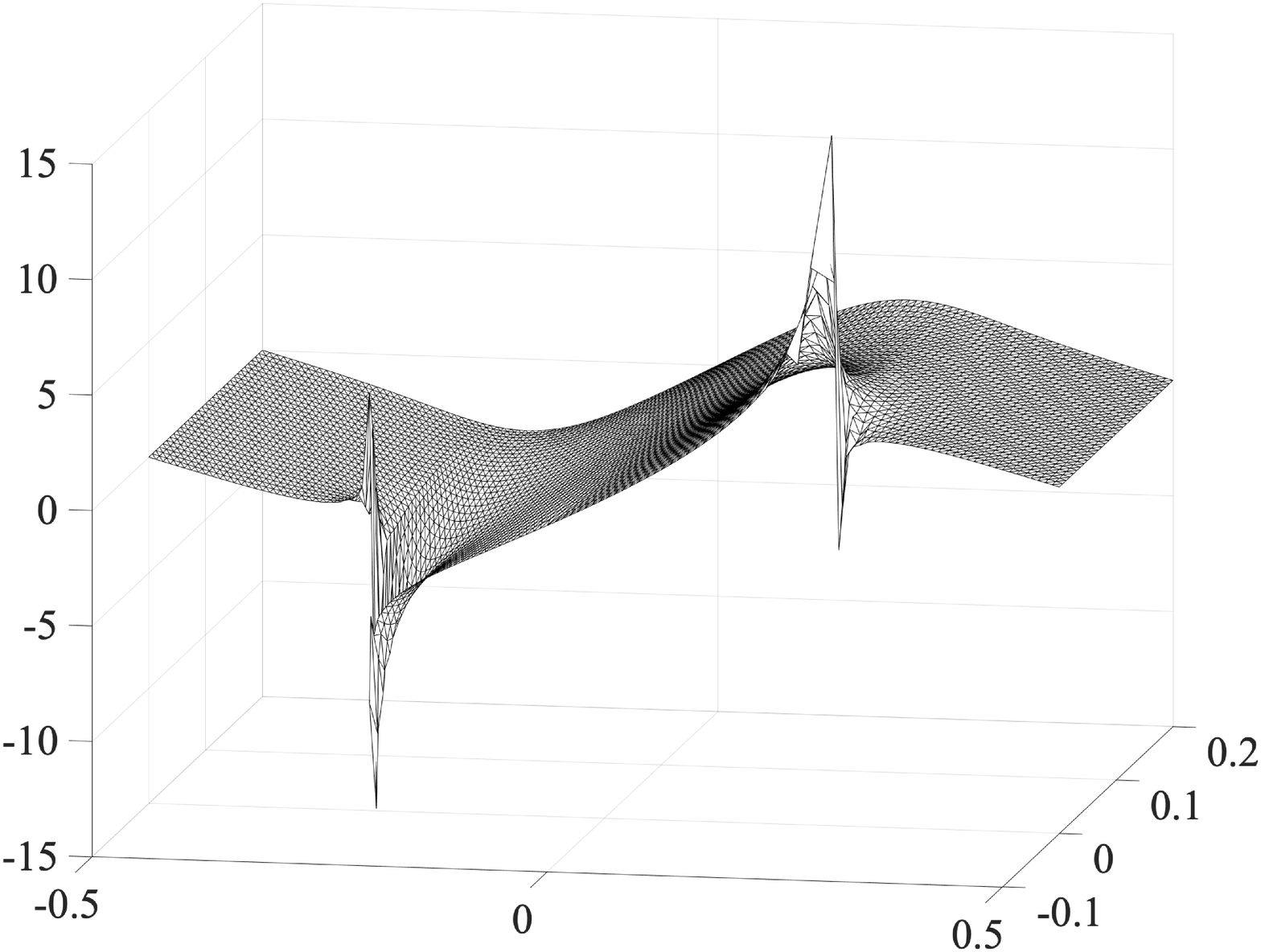}\includegraphics[width=3in]{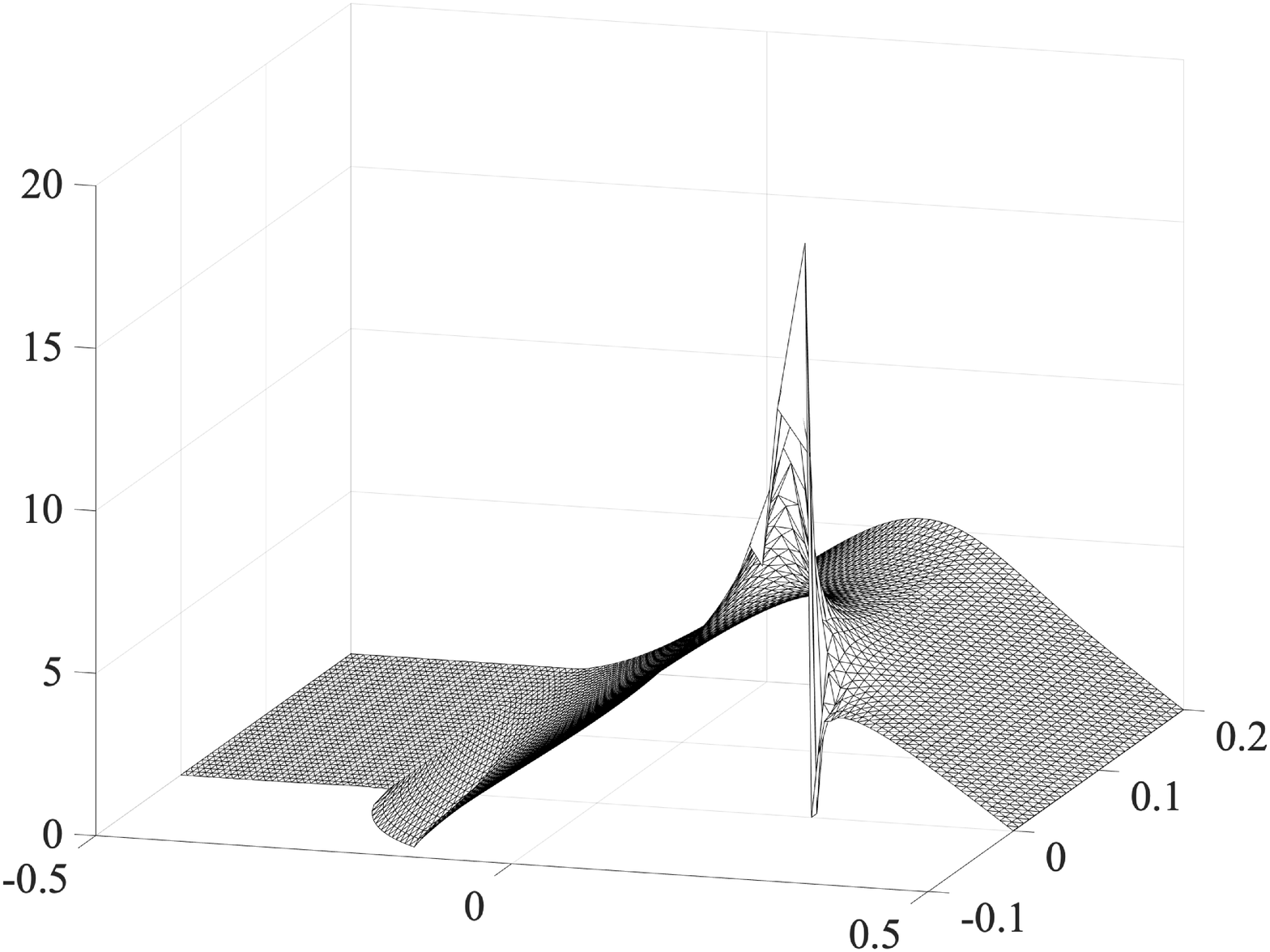}
\end{center}
\caption{Pressure elevation without (left) and with (right) cavitation.\label{fig:cavelev}}
\end{figure}
\begin{figure}[hbt]
\begin{center}
\includegraphics[width=3in]{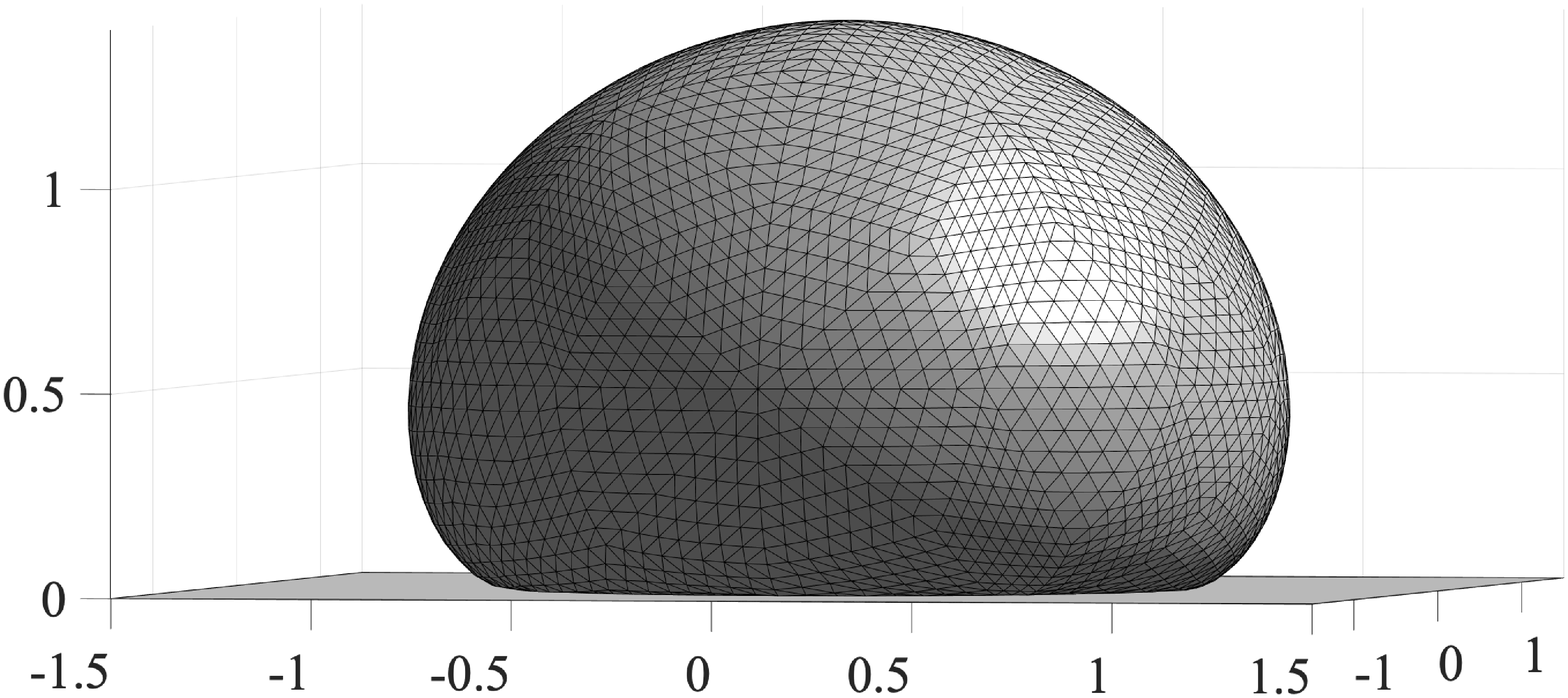}\includegraphics[width=3in]{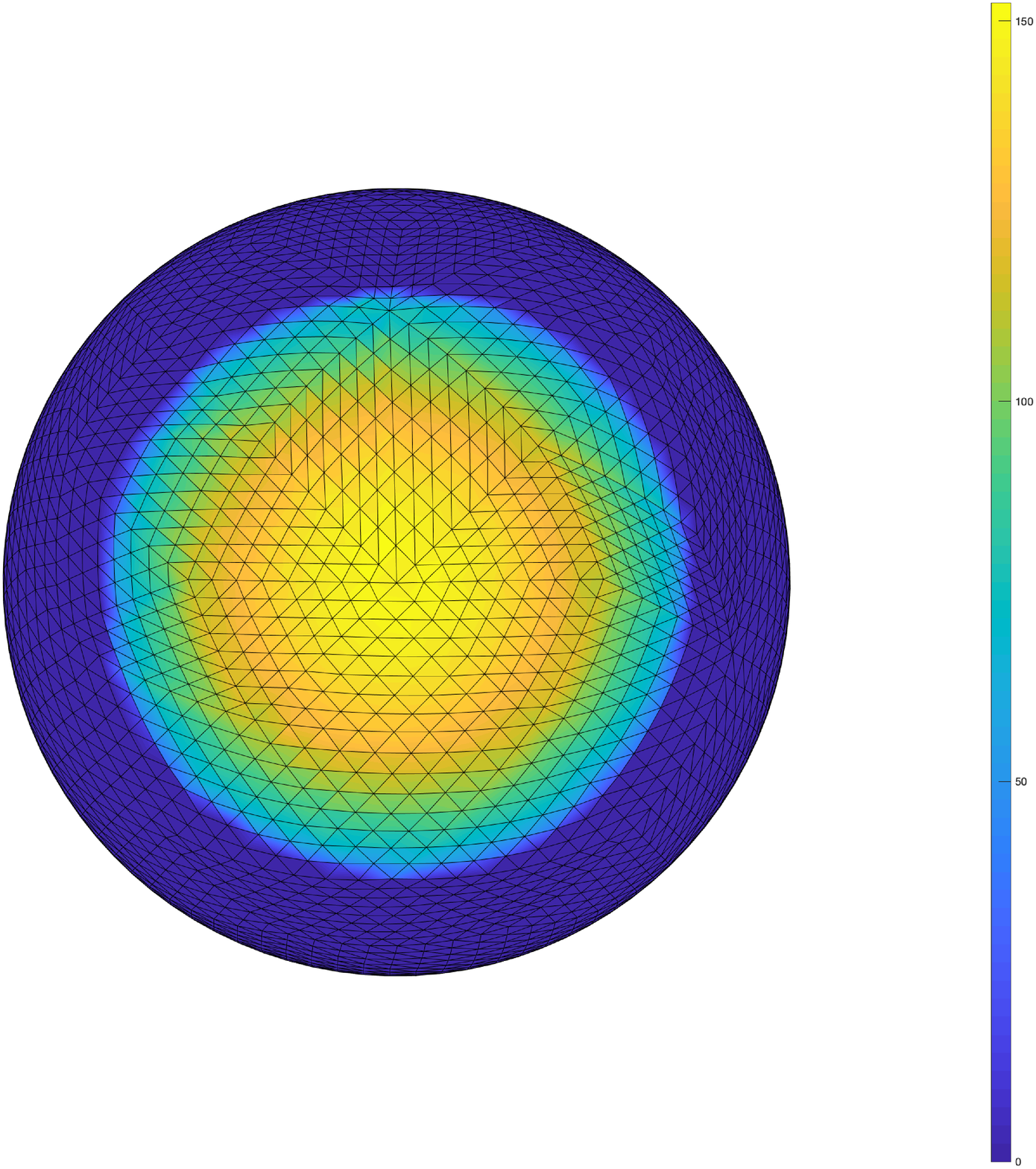}
\end{center}
\caption{Deformations for $\alpha=0$ and associated contact pressure.\label{fig:alpha0}}
\end{figure}
\begin{figure}[hbt]
\begin{center}
\includegraphics[width=3in]{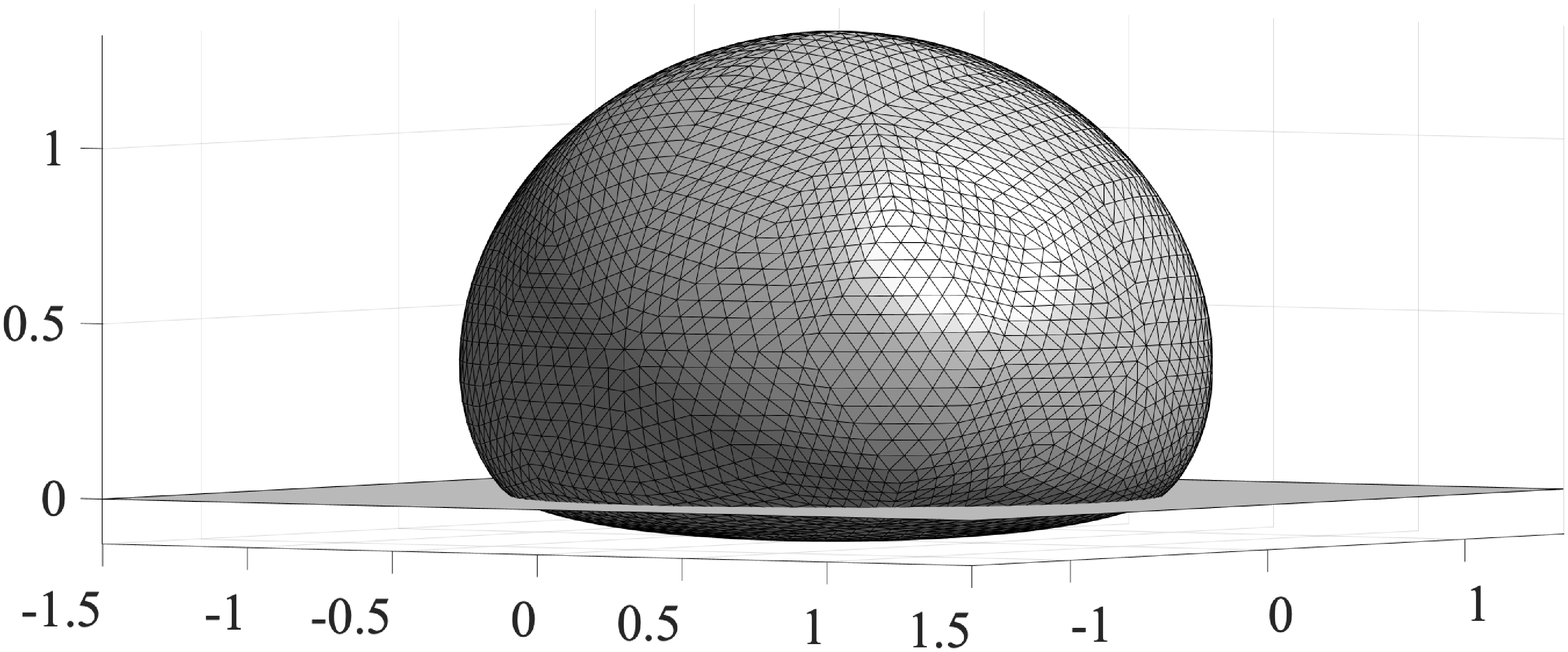}\includegraphics[width=3in]{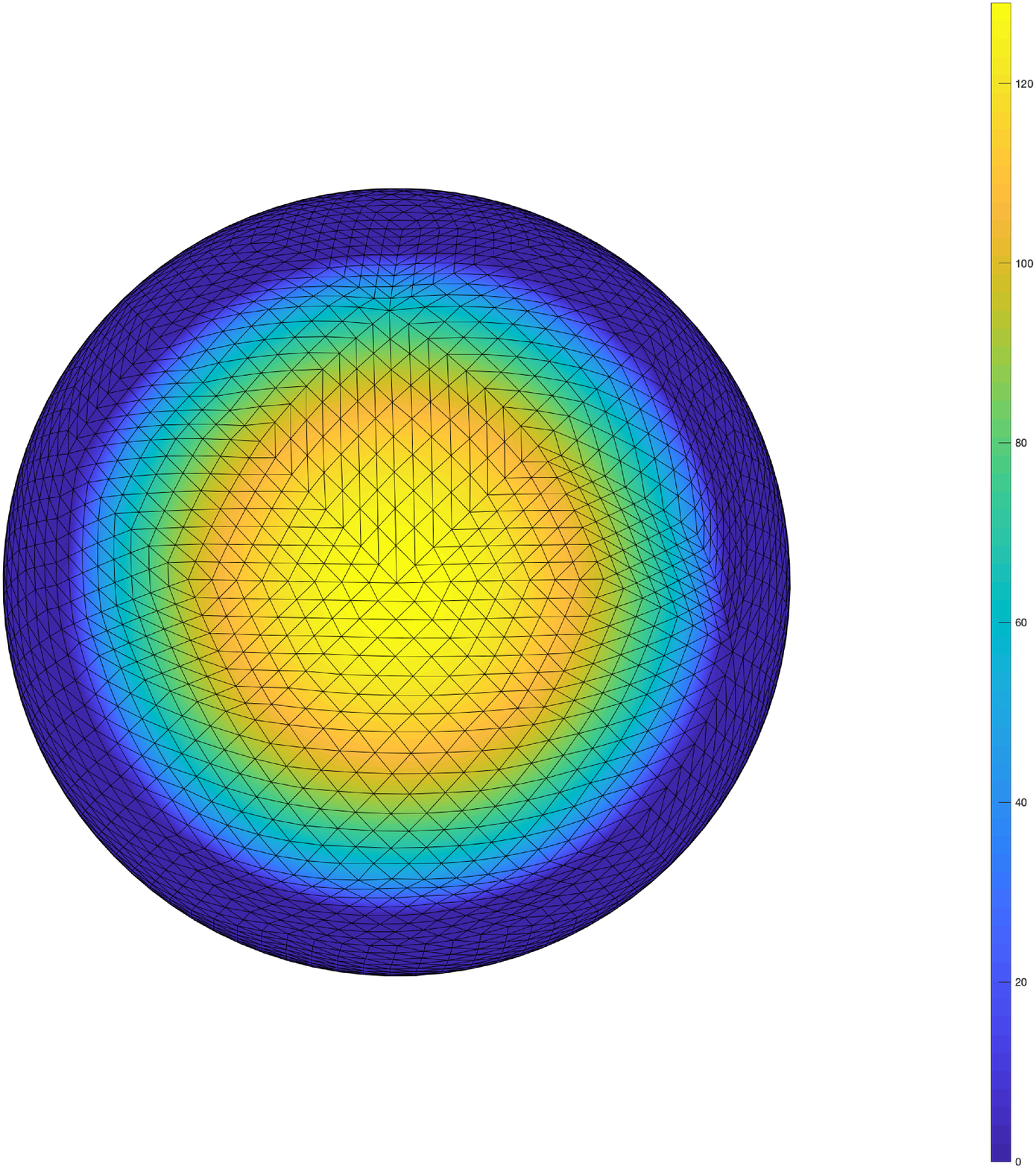}
\end{center}
\caption{Deformations for $\alpha=10^{-3}$ and associated contact pressure.\label{fig:alphaminus3}}
\end{figure}
\begin{figure}[hbt]
\begin{center}
\includegraphics[width=3in]{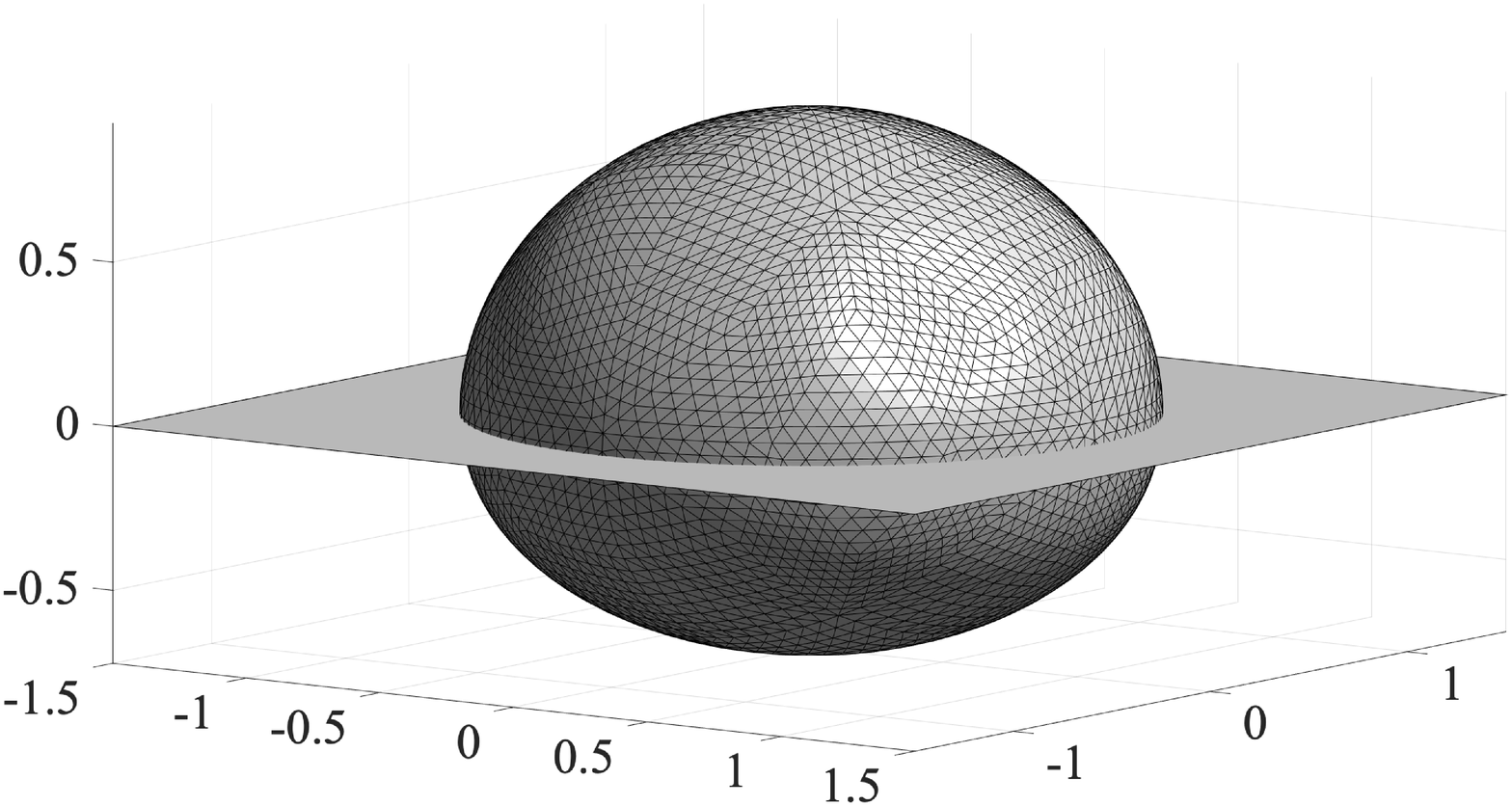}\includegraphics[width=3in]{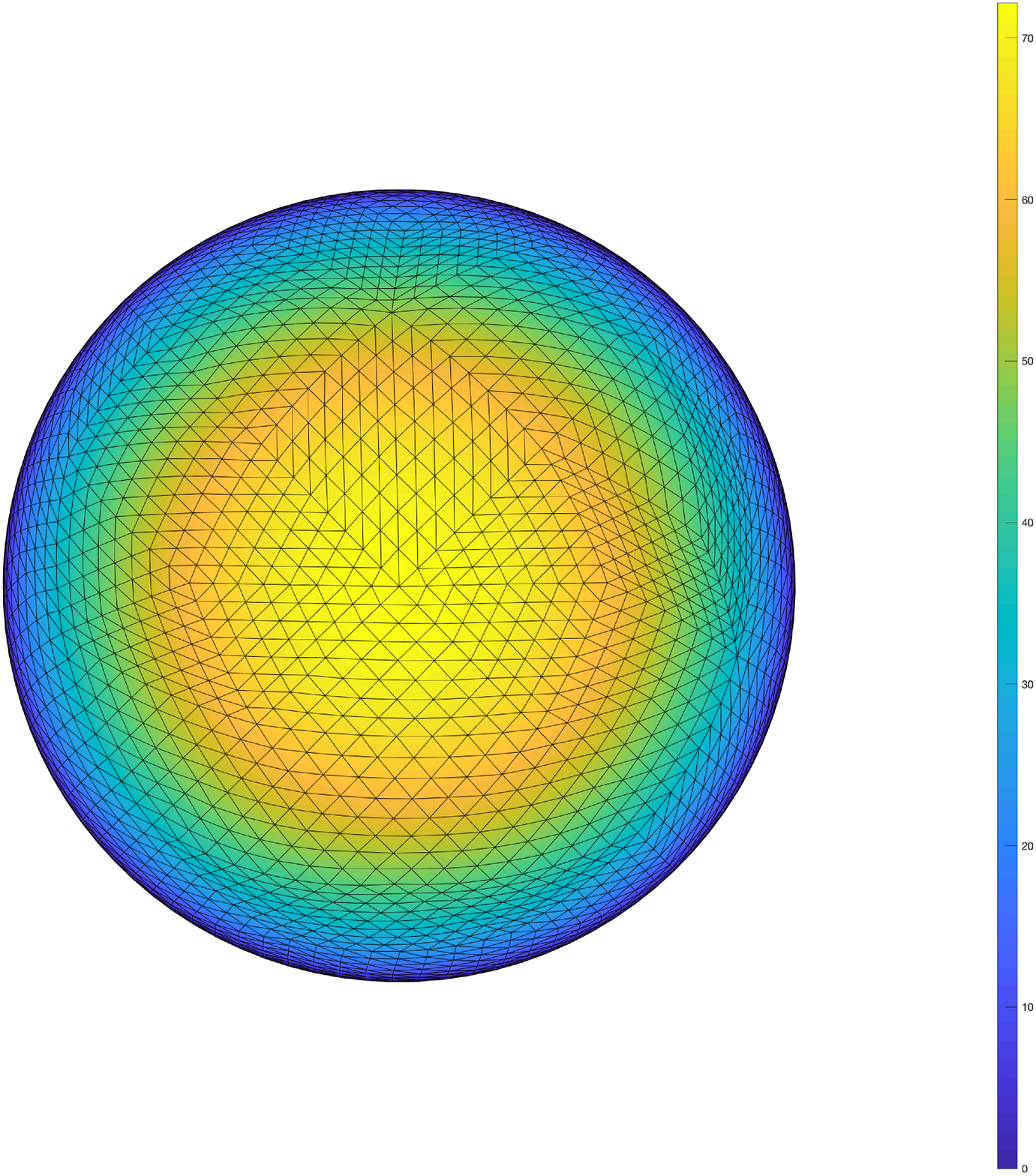}
\end{center}
\caption{Deformations for $\alpha=10^{-2}$ and associated contact pressure.\label{fig:alphaminus2}}
\end{figure}
\begin{figure}[hbt]
\begin{center}
\includegraphics[width=3in]{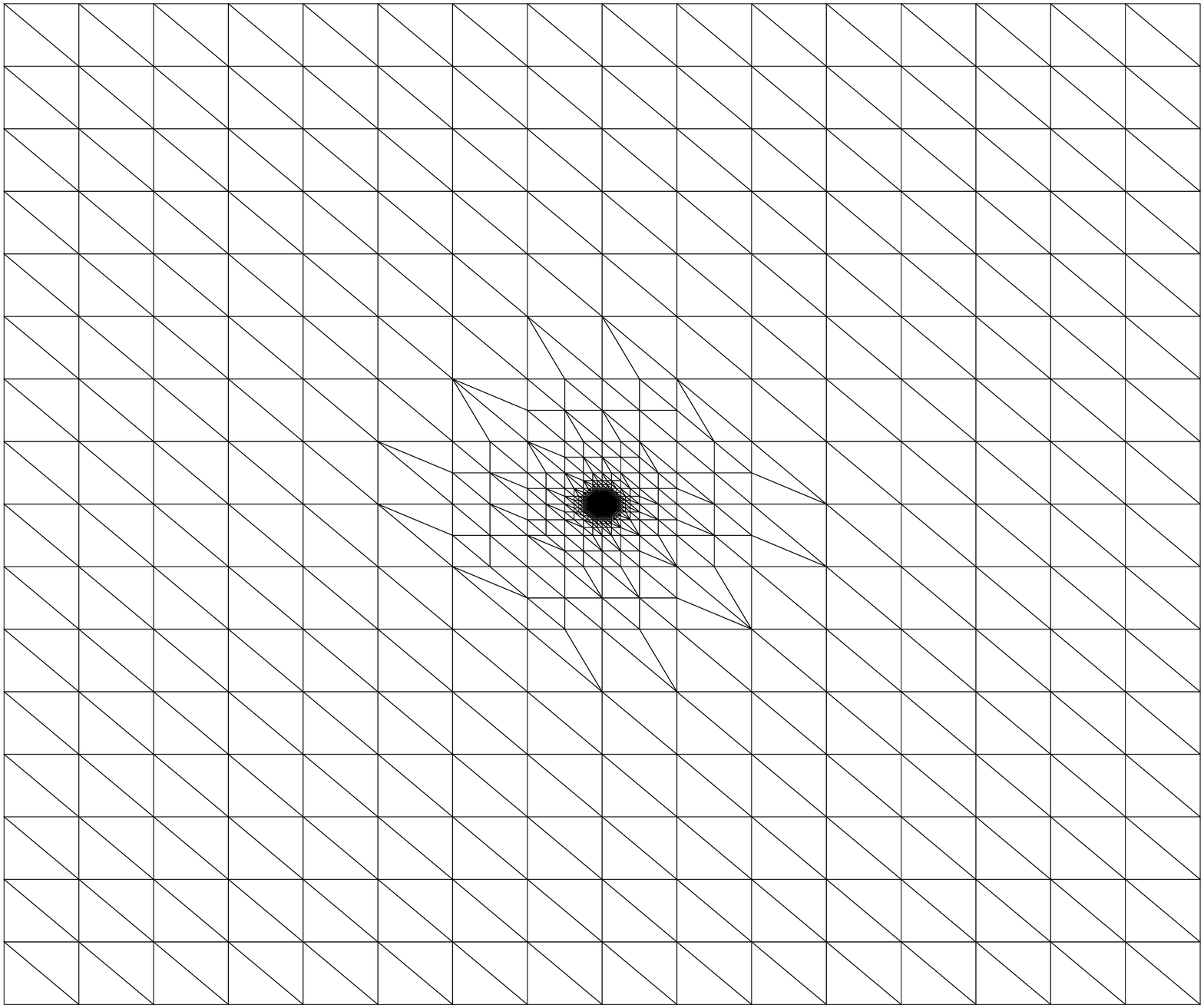}\includegraphics[width=3in]{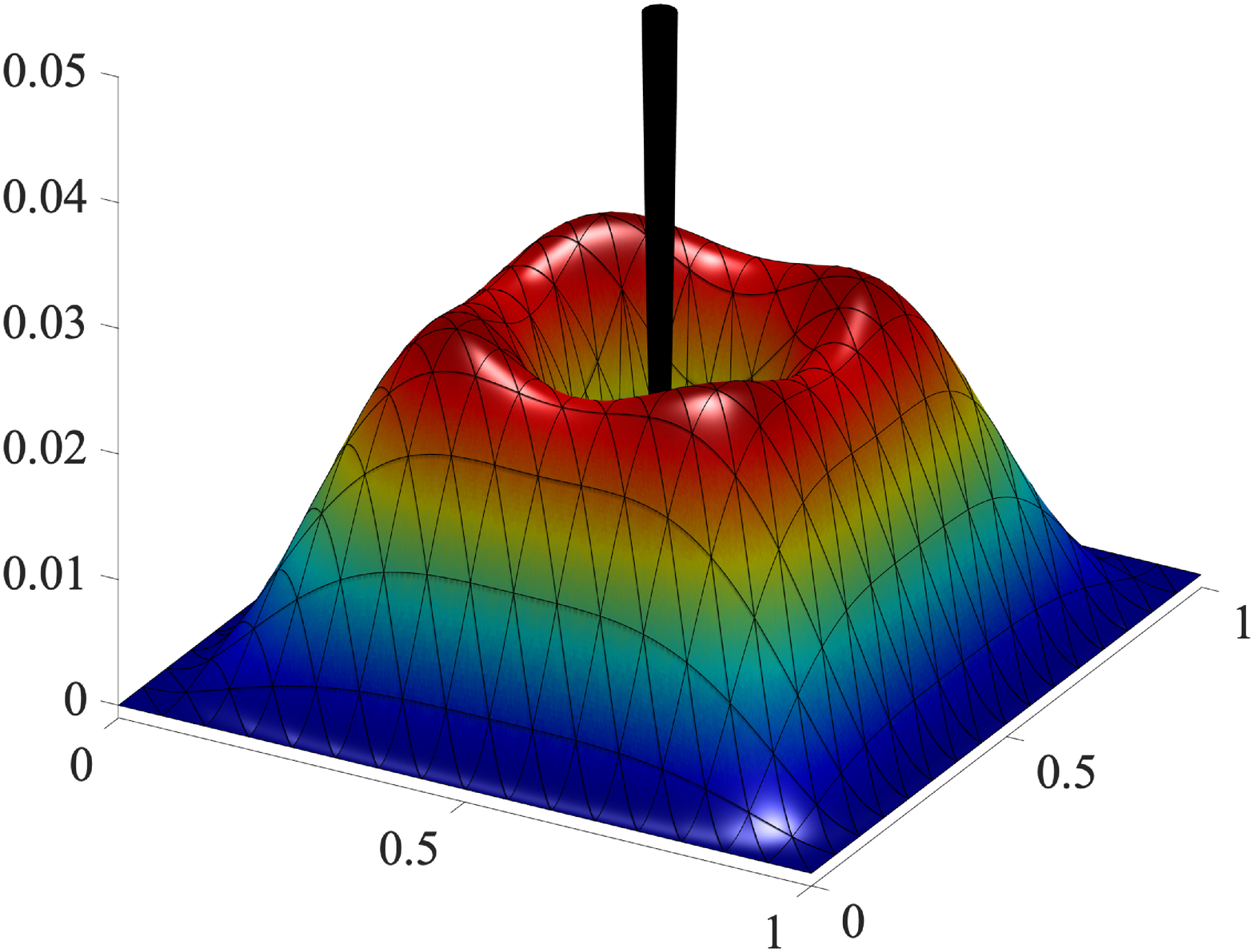}
\end{center}
\caption{Computational mesh and elevation of displacements with obstacle indicated.\label{fig:platemesh}}
\end{figure}
\begin{figure}[hbt]
\begin{center}
\includegraphics[width=4in]{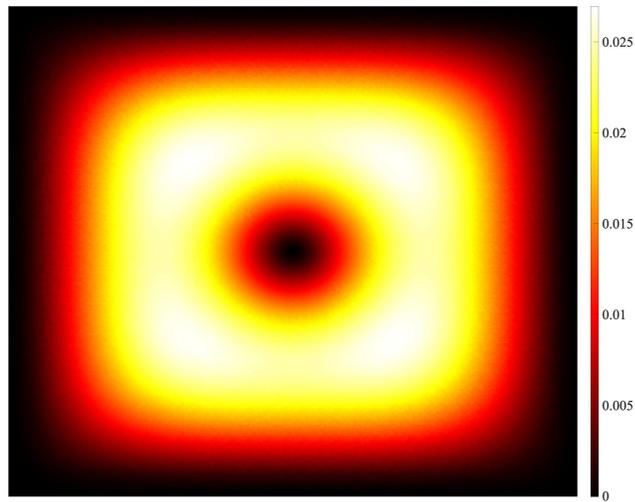}
\end{center}
\caption{Deformation isoplot.\label{fig:plateiso}}
\end{figure}
%\begin{figure}[hbt]
%\begin{center}
%\includegraphics[width=3in]{untitled.eps}
%\end{center}
%\caption{A cohesive interface law enforced by Nitsche's method.\label{fig:discont}}
%\end{figure}
%\begin{figure}[hbt]
%\begin{center}
%\includegraphics[width=3in]{untitled1.eps}
%\end{center}
%\caption{Cohesive interface combined with a contact condition by Nitsche's method.\label{fig:contact}}
%\end{figure}
\end{document}